\documentclass{amsart}

\usepackage{amsmath,amscd,amssymb,amsthm,amsbsy,amscd,stmaryrd}
\usepackage[dvips]{graphicx,color}
\usepackage{epsfig}
\usepackage{mathrsfs}
\input{xy}

\newtheorem{theorem}{Theorem}[section]
\newtheorem{lemma}[theorem]{Lemma}
\newtheorem{proposition}[theorem]{Proposition}
\newtheorem{corollary}[theorem]{Corollary}
\newtheorem{remark}[theorem]{Remark}
\theoremstyle{definition}

\numberwithin{equation}{section}

\begin{document}
\title{Li-Yau-Hamilton estimates and Bakry-Emery Ricci curvature}

\author{Yi Li}
\address{Department of Mathematics, Shanghai Jiao Tong University, 800 Dongchuan Road, Shanghai, 200240 China; Shanghai Center for Mathematical
Sciences, Fudan University, 220 Handan road, Shanghai, 200433 China}
\email{yilicms@gmail.com}

\begin{abstract} In this paper we derive Cheng-Yau, Li-Yau, Hamilton estimates for Riemannian
manifolds with Bakry-Emery Ricci curvature bounded from below, and also global
and local upper bounds, in terms of Bakry-Emery Ricci curvature, for the Hessian of positive and bounded solutions of the weighted heat equation on a closed Riemannian manifold.
\end{abstract}
\maketitle


\section{Introduction}\label{section1}

In a seminal paper \cite{Li-Yau86}, Li and Yau derived the gradient estimate
and Harnack inequality for positive solutions of heat equation on a
complete Riemannian manifold. Li-Yau estimate has been improved and generalized to other
nonlinear equations on a Riemannian manifold, see \cite{AT10, ATW09,
BakryQian00, Brighton13, ChenChen09, Grigoryan09, Hamilton93, Hsu11, KarpLi83, Kotschwar07, LiP12, Li05, Li13, Ma06, Mastrolia10, Mastrolia-Rigoli10, SchoenYau94,
SoupletZhang06, Wang14, Yang08, ZhuLi13} and references therein.

An important generalization is a
diffusion operator
\begin{equation}
\Delta_{V}:=\Delta+\langle V,\nabla\!\ \rangle\label{1.1}
\end{equation}
on a Riemannian manifold $(\mathcal{M},g)$ of dimension $m$, where
$\nabla$ and $\Delta$ are respectively the Levi-Civita connection and Beltrami-Laplace
operator of $g$, and where $V$ is a smooth vector field on $\mathcal{M}$. This
operator is also a special case of
$V$-harmonic map introduced in \cite{ChenJostWang11}. As in \cite{BakryEmery85, ChenJostQiu12}, we introduce Bakey-Emery Ricci tensor fields
\begin{equation}
{\rm Ric}_{V}:={\rm Ric}-\frac{1}{2}\mathscr{L}_{V}g, \ \ \
{\rm Ric}^{n,m}_{V}:={\rm Ric}_{V}-\frac{1}{n-m}V\otimes V\label{1.2}
\end{equation}
for any number $n>m$, where $\mathscr{L}_{V}$ stands for the Lie derivative along
the direction $V$. When $V=\nabla f$, we simply write ${\rm Ric}_{V}$ and
${\rm Ric}^{n,m}_{V}$ as ${\rm Ric}_{f}$ and ${\rm Ric}^{n,m}_{f}$ respectively.

The equation
\begin{equation*}
{\rm Ric}_{V}=\lambda g, \ \ \ \lambda\in{\bf R},
\end{equation*}
is exactly the Ricci soliton
equation, which is one-to-one corresponding to a self-similar solution of
Ricci flow (see, \cite{ChowKnopf04}). A basic example of Ricci solitons is Hamilton's cigar soliton or Witten's balck hole, which is the complete Riemann surface $({\bf R}^{2}, g_{{\rm cs}})$ where
\begin{equation*}
g_{{\rm cs}}:=\frac{dx\otimes dx+dy\otimes dy}{1+x^{2}+y^{2}}.
\end{equation*}
It is easy to see that the scalar curvature of $g_{{\rm cs}}$ is
$4/(1+x^{2}+y^{2})$ and hence the cigar soliton is not Ricci-flat. An important result about the cigar soliton is that it is rotationally symmetric, has positive Gaussian curvature,
is asymptotic to a cyclinder near infinity, and, up to homothety, is the unique
rotationally symmetric gradient Ricci soliton of positive curvature on
${\bf R}^{2}$. Hamilton \cite{Hamilton95} showed that any complete noncompact steady
gradient Ricci soliton with positive Gaussian curvature is a cigar soliton.

To study the Ricci-flat metric on complete noncompact Riemannian manifold, the
author \cite{LY13} found a criterion on Ricci-flat metrics motivated from the
steady gradient Ricci soliton. Moreover, the author introduced a class of
Ricci flow type parabolic differential equation:
\begin{eqnarray}
\partial_{t}g(t)&=&-2{\rm Ric}_{g(t)}+2\alpha_{1}\nabla _{g(t)}\phi(t)
\otimes\nabla_{g(t)}\phi(t)+2\alpha_{2}\nabla^{2}_{g(t)}\phi(t),\label{1.3}\\
\partial_{t}\phi(t)&=&\Delta_{g(t)}\phi(t)+\beta_{1}|\nabla_{g(t)}
\phi(t)|^{2}_{g(t)}+\beta_{2}\phi(t)\label{1.4}
\end{eqnarray}
where $\alpha_{1},\alpha_{2},\beta_{1},\beta_{2}$ are given constants. Note
that the equation (\ref{1.3}) can be written as
\begin{equation}
\partial_{t}g(t)=-2{\rm Ric}^{n,m}_{V(t)}\label{1.5}
\end{equation}
for some suitable constants $\alpha_{1}, \alpha_{2}, n$, where $V(t):=
\nabla \phi(t)$. Hence the Bakry-Emery-Ricci curvature naturally
appears in \cite{LY13}. Under some hypotheses on initial data and constants $\alpha_{i}, \beta_{i}$, the author proved the short time existence and Berstein's type estimates for
(\ref{1.3})--(\ref{1.4}) in \cite{LY13}.

Another important relation between Bakry-Emery-Ricci curvature is the study
of Killing vector fields. The authors in \cite{LL11} investigated the gradient
flow for the functional
\begin{equation}
\mathcal{I}(X):=\int_{\mathcal{M}}|\mathscr{L}_{X}g|^{2}dV.\label{1.6}
\end{equation}
on the space of smooth vector fields. The critical point $X$ of $\mathcal{I}$ satisfies
\begin{equation}
\Delta X^{i}+\nabla^{i}{\rm div}(X)+R^{i}{}_{j}X^{j}=0.\label{1.7}
\end{equation}
We then in \cite{LL11} introduced a flow
\begin{equation}
\partial_{t}X_{t}=\Delta X_{t}+\nabla{\rm div}(X_{t})+{\rm Ric}(X_{t}), \ \ \
X_{0}:=X,\label{1.8}
\end{equation}
to study the existence of nonzero Killing vector fields on a closed positively
curved manifold. Actually, we showed that

\begin{theorem}\label{t1.1}{\bf (Li-Liu \cite{LL11}, 2011)} Suppose that $(\mathcal{M},g)$
is a closed and orientable Riemannian manifold. If $X$ is a smooth vector field, there exists a unique smooth solution $X_{t}$ to the flow (\ref{1.8}) for all time $t$. As $t$ goes
to infinity, the vector field $X_{t}$ converges uniformly to a Killing
vector field $X_{\infty}$.
\end{theorem}

The above theorem does {\it not} give a nontrivial Killing vector field, since
Bochner's theorem implies that there is no nontrivial Killing vector field on a
closed Riemannian manifold with negative Ricci curvature. For more information
on the flow (\ref{1.8}), we refer to the paper \cite{LL11}. In the same paper \cite{LL11}, we give the second criterion on the existence
of Killing vector fields. This observation is based on the following identity
\begin{equation*}
\int_{\mathcal{M}}\left[(\mathscr{L}_{X}g)(X,X)+\frac{1}{2}{\rm div}(X)
|X|^{2}\right]dV=0
\end{equation*}
where $X$ is a smooth vector field on $\mathcal{M}$. A quite simple argument
showed that

\begin{theorem}\label{t1.2}{\bf (Li-Liu \cite{LL11}, 2011)} A smooth vector field $X$ on a closed and orientable
Riemannian manifold $(\mathcal{M},g)$ is Killing if and only if
\begin{equation}
0=\Delta X+\nabla{\rm div}(X)+{\rm Ric}_{-2X}(X)+\frac{1}{2}{\rm div}(X)X.
\label{1.9}
\end{equation}
\end{theorem}

The third criterion in \cite{LL11} is based on Lott's observation
\cite{Lott03}:
\begin{equation*}
\int_{\mathcal{M}}|\mathscr{L}_{X}g|^{2}e^{-f}dV=-
\int_{\mathcal{M}}\left\langle X,\Delta_{f}X+\nabla{\rm div}_{f}(X)
+{\rm Ric}_{f}(X)\right\rangle e^{f}dV.
\end{equation*}
The we proved the following

\begin{theorem}\label{t1.3}{\bf (Li-Liu \cite{LL11}, 2011)} Given any smooth function $f$ on a closed and orientable
Riemannian manifold $(\mathcal{M},g)$. A smooth vector field $X$ is Killing if and
only if it satisfies
\begin{equation}
0=\Delta X^{i}+\nabla^{i}{\rm div}(X)+R^{i}{}_{j}X^{j}+\nabla_{j}
f(\mathscr{L}_{X}g)^{ij}.\label{1.10}
\end{equation}
In particular, $X$ is Killing if and only if
\begin{equation}
0=\Delta X^{i}+\nabla^{i}{\rm div}(X)+R^{i}{}_{j}X^{j}
+\nabla_{j}{\rm div}(X)(\mathscr{L}_{X}g)^{ij}.\label{1.11}
\end{equation}
\end{theorem}

Those elliptic equations (\ref{1.9})--(\ref{1.10}) can be made into
the corresponding parabolic equations which may play well in the study
of the existence of nontrivial Killing vector fields and moreover in the
study of Hopf's conjecture and Yau's problem.

We now state our main results in this paper. The first three results are
about Cheng-Yau estimates for complete Riemannian manifold with
${\rm Ric}^{n,m}_{V}$ bounded from below.

\begin{theorem}\label{t1.4} Let $(\mathcal{M},g)$ be a compact $m$-dimensional
Riemannian manifold with ${\rm Ric}^{n,m}_{V}\geq-K$, where $K\geq0$ is
a constant. If $u$ is a solution of $\Delta_{V}u=0$ which is bounded from
below, then
\begin{equation}
|\nabla u|\leq \sqrt{(n-1)K}\left(u-\inf_{\mathcal{M}}u\right).\label{1.12}
\end{equation}
In particular, if ${\rm Ric}^{n,m}_{V}\geq0$, then every positive solution of
$\Delta_{V}u=0$ must be constant.
\end{theorem}

\begin{theorem}\label{t1.5} Let $(\mathcal{M},g)$ be a complete $m$-dimensional Riemannian
manifold with ${\rm Ric}^{n,m}_{V}\geq-(n-1)K$ where $K\geq0$ is a constant. If $u$ is a positive solution of $\Delta_{V}u=0$ on $\mathcal{M}$, for any $r>0$, we have
\begin{equation}
\sup_{B(x,r/2)}\frac{|\nabla u|}{u}\leq 8(n-1)\left(\frac{1}{r}
+\sqrt{K}\right).\label{1.13}
\end{equation}
\end{theorem}

\begin{corollary}\label{c1.6} Let $(\mathcal{M},g)$ be a complete $m$-dimensional Riemannian manifold with ${\rm Ric}^{n,m}_{V}\geq-(n-1)K$ where $K\geq0$ is a constant.
\begin{itemize}

\item[(i)] If $(\mathcal{M},g)$ is noncompact and $u$ is a positive solution of $\Delta_{V}u=0$ on $\mathcal{M}$, then
    \begin{equation}
    \sup_{\mathcal{M}}\frac{|\nabla u|}{u}\leq 8(n-1)\sqrt{K}.\label{1.14}
    \end{equation}

\item[(ii)] If $u$ is a solution of $\Delta_{V}u=0$ on a geodesic ball $B(x,r)$, then
\begin{equation}
\sup_{B(x,r/2)}|\nabla u|\leq 16(n-1)\left(\frac{1}{r}+\sqrt{K}\right)
\sup_{B(x,r)}|u|.\label{1.15}
\end{equation}

\item[(iii)] If $u$ is a positive solution of $\Delta_{V}u=0$ on a geodesic ball $B(x,r)$, then
\begin{equation}
\sup_{B(x,r/2)}u\leq e^{8(n-1)(1+2r\sqrt{K})}
\inf_{B(x,r/2)}u.\label{1.16}
\end{equation}
\end{itemize}
\end{corollary}

When $V\equiv0$, those estimates are the classical results
\cite{ChengYau75, SchoenYau94}. If $V$ is gradient, the above results
reduce to those of \cite{Li05}.

Recall that \cite{Grigoryan09} a triple $(\mathcal{M}, g, \mu)$ is called a weighted Riemannian manifold, if $(\mathcal{M},g)$ is a Riemannian manifold and $\mu$ is a measure on
$\mathcal{M}$ with a smooth positive density function $f$ (that is, $d\mu=fd
V_{g}$). The {\it weighted divergence} and the {\it weighted Laplace operator} are defined by
\begin{equation*}
{\rm div}_{\mu}=\frac{1}{f}{\rm div}(f\!\ ), \ \ \ \Delta_{\mu}:=
{\rm div}_{\mu}\circ\nabla
\end{equation*}
respectively, where $\nabla$ is the Levi-Civita connection of $g$. There are two
examples of $\Delta_{\mu}$:
\begin{itemize}

\item[(a)] When $V=\nabla f$, the operator $\Delta_{V}$ is exactly the weighted
Laplace operator of the weighted Riemannian manifold $(\mathcal{M}, g,\mu)$
where $\mu=e^{f}dV_{g})$. Indeed,
\begin{equation*}
\Delta_{\mu}=\frac{1}{e^{f}}{\rm div}(e^{f}\nabla\!\ )=\frac{1}{e^{f}}
\left(e^{f}\Delta+\langle\nabla e^{f},\nabla\!\ \rangle\right)=\Delta+\langle
\nabla f,\nabla\!\ \rangle=:\Delta_{f}.
\end{equation*}

\item[(b)] In \cite{Mastrolia-Rigoli10}, the authors introduced a diffusion-type operator
\begin{equation*}
L=\frac{1}{B}{\rm div}(A\nabla\!\ )
\end{equation*}
where $A, B$ are some sufficiently smooth positive functions on $\mathcal{M}$. Set
\begin{equation*}
\tilde{g}:=\frac{B}{A}g, \ \ \ d\tilde{\mu}:=BdV_{g}.
\end{equation*}
Then $L$ is the weighted Laplace operator of the weighted Riemannian manifold $(\mathcal{M},
\tilde{g},\tilde{\mu})$ since
\begin{equation*}
\widetilde{\Delta}_{\tilde{\mu}}=
{\rm div}_{\tilde{\mu}}\circ\widetilde{\nabla}
=\frac{1}{B}{\rm div}\left(B\frac{A}{B}\nabla\!\ \right)=L.
\end{equation*}

\end{itemize}
In both cases, $\Delta_{f}$ or $L$ can be viewed as the special case of
$\Delta_{V}$ on some Riemannian manifold.

\begin{theorem}\label{t1.7} Let $(\mathcal{M},g)$ be a complete $m$-dimensional Riemannian
manifold with ${\rm Ric}^{n,m}_{V}\geq-(n-1)K(1+d^{2})^{\delta/2}$, where
$K\geq0$, $\delta<4$, and $d$ denotes the distance function from a fixed
point. If $F\in C^{1}({\bf R})$ and $u\in C^{3}(\mathcal{M})$ is a global solution of
\begin{equation*}
\Delta_{V}u=F(u)
\end{equation*}
with
\begin{equation*}
|u|\leq D(1+d)^{\nu}, \ \ \ F'(u)\geq(n-1)K(1+d^{2})^{\delta/2}
\end{equation*}
on $\mathcal{M}$ for some constants $D>0$ and $0<\nu<\min\{1,1-\frac{\delta}{4}\}$, then $u$ must be constant.
\end{theorem}

Theorem \ref{t1.7} generalized the similar result in
\cite{Mastrolia10, Mastrolia-Rigoli10}. The proof is based on variants of $V$-Bochner-Weitzenb\"ock formula stated in Section \ref{section2}.

Next three estimates are about Li-Yau gradient estimates for positive solutions
of weighted heat type equation on a complete Riemannian manifold, and extend the corresponding
results in \cite{ZhuLi13} from heat type equation to weighted heat type equation.

\begin{theorem}\label{t1.8} Let $(\mathcal{M},g)$ be a compact $m$-dimensional Riemannian
manifold with ${\rm Ric}^{n,m}_{V}\geq0$. Suppose that the boundary $\partial\mathcal{M}$ of $\mathcal{M}$ is convex whenever $\partial\mathcal{M}\neq\emptyset$. Let $u$ be a positive solution of
\begin{equation*}
\left(\Delta_{V}-\partial_{t}\right)u=au\ln u
\end{equation*}
on $\mathcal{M}\times(0,T]$ for some constant $a$, with Neumann boundary condition $\frac{\partial u}{\partial\nu}=0$ on $\partial\mathcal{M}\times(0,T]$.
\begin{itemize}

\item[(1)] If $q\leq0$ then
\begin{equation*}
\frac{|\nabla u|^{2}}{u^{2}}
-\frac{u_{t}}{u}-a\ln u\leq\frac{n}{2t}-\frac{na}{2}
\end{equation*}
on $\mathcal{M}\times(0,T]$.

\item[(2)] If $a\geq0$ then
\begin{equation*}
\frac{|\nabla u|^{2}}{u^{2}}
-\frac{u_{t}}u-a\ln u\leq\frac{n}{2t}.
\end{equation*}

\end{itemize}

\end{theorem}

\begin{theorem}\label{t1.9} Let $(\mathcal{M},g)$ be a complete manifold with boundary $\partial\mathcal{M}$. Assume that $p\in\mathcal{M}$ and the geodesic ball $B(p,2R)$ does not intersect
$\partial\mathcal{M}$. We denote by $-K(2R)$ with $K(2R)\geq0$, a lower bound
of ${\rm Ric}^{n,m}_{V}$ on the ball $B(p,2R)$. Let $q$ be a function defined on $\mathcal{M}\times[0,T]$ which is $C^{2}$ in the $x$ variable and $C^{1}$ in the $t$ variable. Assume that
\begin{equation*}
\Delta_{V}q\leq\theta(2R), \ \ \ |\nabla q|\leq\gamma(2R)
\end{equation*}
on $B(p,2R)\times[0,T]$ for some constants $\theta(2R)$ and $\gamma(2R)$. If $u$ is a positive solution of the equation
\begin{equation*}
\left(\Delta_{V}-q-\partial_{t}\right)u=au\ln u
\end{equation*}
on $\mathcal{M}\times(0,T]$ for some constant $a$, then for any $\alpha>1$ and
$\epsilon\in(0,1)$, on $B(p,R)$, $u$ satisfies the following estimates:
\begin{itemize}

\item[(1)] for $a\geq0$, we have
\begin{eqnarray*}
|\nabla f|^{2}-\alpha f_{t}-\alpha q-\alpha af&\leq&\frac{n\alpha^{2}}{2(1-\epsilon)t}
+\frac{(A+\gamma)n\alpha^{2}}{2(1-\epsilon)}+\frac{n^{2}\beta^{4}C^{2}_{1}}{4\epsilon
(1-\epsilon)(\beta-1)R^{2}}\\
&&+ \ \frac{n\alpha^{2}[K+a(\alpha-1)]}{(1-\epsilon)(\alpha-1)}
+\left(\frac{[\alpha\theta+(\alpha-1)\gamma]n\alpha^{2}}{2(1-\epsilon)}
\right)^{1/2}.
\end{eqnarray*}

\item[(2)] for $a\leq0$, we have
\begin{eqnarray*}
|\nabla f|^{2}-\alpha f_{t}-\alpha q-\alpha af&\leq&\frac{n\alpha^{2}}{2(1-\epsilon)t}
+\frac{(A+\gamma)n\alpha^{2}}{2(1-\epsilon)}+\frac{n^{2}\beta^{4}C^{2}_{1}}{4\epsilon
(1-\epsilon)(\beta-1)R^{2}}\\
&&+ \ \frac{n\alpha^{2}[K-\frac{a}{2}a(\alpha-1)]}{(1-\epsilon)(\alpha-1)}
+\left(\frac{[\alpha\theta+(\alpha-1)\gamma]n\alpha^{2}}{2(1-\epsilon)}
\right)^{1/2}.
\end{eqnarray*}

\end{itemize}
Here $f:=\ln u$ and $A=[2C^{2}_{1}+(n-1)C^{2}_{1}(1+R\sqrt{K})
+C_{2}]/R^{2}$ for some positive constants $C_{1}, C_{2}$.
\end{theorem}

\begin{corollary}\label{c1.10} If $(\mathcal{M},g)$ is a complete noncompact
Riemannian manifold without boundary and ${\rm Ric}^{n,m}_{V}\geq-K$
on $\mathcal{M}$, then any positive solution $u$ of the equation
\begin{equation*}
\partial_{t}u=\Delta_{V}u
\end{equation*}
on $\mathcal{M}\times(0,T]$ satisfies
\begin{equation}
\frac{|\nabla u|^{2}}{u^{2}}-\alpha\frac{u_{t}}{u}
\leq\frac{n\alpha^{2}K}{\alpha-1}+\frac{n\alpha^{2}}{2t}\label{1.17}
\end{equation}
for any $\alpha>1$.
\end{corollary}

As pointed in \cite{SchoenYau94}, the estimate (\ref{1.17}) still
holds for any closed Riemannian manifold with ${\rm Ric}^{n,m}_{V}\geq-K$.

Thirdly, we derive Hamilton's Harnack inequality for $\Delta_{V}$ operator.
Setting $V\equiv0$ in Theorem \ref{t1.11}, we obtain the classical result of
Hamilton \cite{Hamilton93}. Later Kotschwar \cite{Kotschwar07} extended Hamilton's gradient estimate to complete noncompact Riemannian manifold. Li \cite{Li13} proved Hamilton's
gradient estimate for $\Delta_{V}$ where $V=-\nabla\phi$, both in
compact case
and noncompact case.

\begin{theorem}\label{t1.11} Suppose that $(\mathcal{M},g)$ is a compact Riemannian manifold with ${\rm Ric}_{V}\geq-K$ where $K\geq0$. If $u$ is a
solution of $\partial_{t}u=\Delta_{V}u$ with $0<u\leq A$ on $\mathcal{M}\times(0,T]$, then
\begin{equation}
\frac{|\nabla u|^{2}}{u^{2}}\leq
\left(\frac{2K}{e^{2Kt}-1}+2K\right)
\ln\frac{A}{u}\leq\left(\frac{1}{t}+2K\right)\ln\frac{A}{u}\label{1.18}
\end{equation}
on $\mathcal{M}\times(0,T]$.
\end{theorem}

As a consequence of Theorem \ref{t1.11}, we generalize a result in \cite{Brighton13, Li13} about the Liouville theorem.

\begin{corollary}\label{c1.12} Suppose that $(\mathcal{M},g)$ is a compact Riemannian
manifold with ${\rm Ric}_{V}\geq-K$ where $K\geq0$. If $u$ is a positive solution
of $\Delta_{V}u=0$ on $\mathcal{M}$ then
\begin{equation}
|\nabla\ln u|^{2}
\leq2K\ln\frac{\sup_{\mathcal{M}}u}{u}\label{1.19}
\end{equation}
In particular if ${\rm Ric}_{V}\geq0$ every bounded solution $u$ satisfying $\Delta_{V}u=0$ must be constant.
\end{corollary}

A local version of Hamilton's estimate was proved
by Souplet and Zhang \cite{SoupletZhang06} for $\Delta$, while by Arnaudon, Thalmaier, and Wang \cite{ATW09} for the general operator $\Delta_{V}$. A probabilistic proof
of Hamilton's estimates for $\Delta$ and $\Delta_{V}$ with $V=-\nabla\phi$
can be found in \cite{AT10, Li13}. In this paper we give a geometric proof
of Hamilton's estimate for Witten's Laplacian, following the method in \cite{Kotschwar07} together with Karp-Li-Grigor'yan maximum principle for complete manifolds.

\begin{theorem}\label{t1.13} Suppose that $(\mathcal{M},g)$ is a complete noncompact Riemannian manifold with ${\rm Ric}^{n,m}_{f}\geq-K$ where $K\geq0$. If $u$ is a solution of $\partial_{t}u=
\Delta_{f}u$ with $0<u\leq A$ on $\mathcal{M}\times(0,T]$, then
\begin{equation}
\frac{|\nabla u|^{2}}{u^{2}}\leq\left(\frac{2K}{e^{2Kt}-1}+2K\right)
\ln\frac{A}{u}\leq\left(\frac{1}{t}+2K\right)\ln\frac{A}{u}\label{1.20}
\end{equation}
on $\mathcal{M}\times(0,T]$.
\end{theorem}

We compare other Hamilton's estimates with (\ref{1.20}). In our
geometric proof we require the curvature condition ${\rm Ric}^{n,m}_{f}\geq-K$
in order to use the Bakry-Qian's Laplacian comparison theorem without
any additional requirement on the potential function $f$. If we use the curvature condition ${\rm Ric}_{f}\geq-K$ in our geometric proof, then some conditions on $f$ would be
required (see \cite{ChenJostQiu12, WeiWylie09}). A probabilistic proof
of Li \cite{Li13} shows a similar estimate
\begin{equation*}
\frac{|\nabla u|^{2}}{u^{2}}\leq\left(\frac{2}{t}+2K\right)\ln\frac{A}{u}
\end{equation*}
where $0<u\leq A$ on $\mathcal{M}\times(0,T]$ and ${\rm Ric}_{f}\geq-K$.

In the last part, we generalize Hessian estimates for positive solutions of the heat
equation in \cite{HanZhang12} to
these of the weighted heat equation.

\begin{theorem}\label{t1.14} Let $(\mathcal{M},g)$ be a closed $m$-dimensional
Riemannian manifold with ${\rm Ric}^{n,m}_{V}\geq-K$ where $K\geq0$.
\begin{itemize}

\item[(a)] If $u$ is a solution of $\partial_{t}u=\Delta_{V}u$ in $\mathcal{M}
\times(0,T]$ and $0<u\leq A$, then
\begin{equation}
\nabla^{2}u\leq\left(B+\frac{5}{t}\right)u\left(1+\ln\frac{A}{u}\right)g
\label{1.21}
\end{equation}
in $\mathcal{M}\times(0,T]$, where $B=10m^{3/2}n\mathcal{K}_{V}$,
\begin{equation*}
\mathcal{K}_{V}:=K_{1}+K_{2}+\sqrt{(K_{1}+K_{2})K+K_{2}
+K_{1}\sup_{\mathcal{M}}|V|^{2}}
\end{equation*}
with $K_{1}=\max_{\mathcal{M}}(|{\rm Rm}|+|{\rm Ric}_{V}|)$ and $K_{2}
=\max_{\mathcal{M}}|\nabla{\rm Ric}_{V}|$.

\item[(b)] If $u$ is a solution of $\partial_{t}u=\Delta_{V}u$ in $Q_{R,T}(x_{0},t_{0})$ and $0<u\leq A$, then
    \begin{equation}
    \nabla^{2}u\leq C_{1}\left(\frac{1}{T}+\frac{1+R\sqrt{K}}{R^{2}}+B\right)u
    \left(1+\ln\frac{A}{u}\right)^{2}g\label{1.22}
    \end{equation}
    in $Q_{R/2,T/2}(x_{0},t_{0})$, where $B=C_{2}m^{5/2}n^{2}
    \mathcal{K}_{V}$ and $C_{1}, C_{2}$ are positive universal constants.

\end{itemize}

\end{theorem}

\section{$V$-Bochner-Weitzenb\"ock formula and its applications}\label{section2}

To prove Li-Yau-Hamilton estimates for $V$-weighted equation, we need the following Bochner-Weitzenb\"ock formula for $V$-Laplace operator.

\begin{lemma}\label{l2.1} Given a smooth vector field $V$ on a Riemannian 
manifold $(\mathcal{M},g)$. For any smooth function $u$ on $\mathcal{M}$, we have
\begin{equation}
\frac{1}{2}\Delta_{V}|\nabla u|^{2}
=|\nabla^{2}u|^{2}+{\rm Ric}_{V}(\nabla u,\nabla u)
+\langle\nabla\Delta_{V}u,\nabla u\rangle.\label{2.1}
\end{equation}
In particular, we have
\begin{eqnarray}
\frac{1}{2}\Delta_{V}|\nabla u|^{2}
&\geq&\frac{1}{n}(\Delta_{V}u)^{2}
+{\rm Ric}^{n,m}_{V}(\nabla u,\nabla u)+\left\langle\nabla\Delta_{V}u,\nabla u
\right\rangle\label{2.2},\\
\frac{1}{2}\Delta_{V}|\nabla u|^{2}&\geq&
|\nabla^{2}u|^{2}+{\rm Ric}^{n,m}_{V}(\nabla u,\nabla u)+\langle
\nabla\Delta_{V}u,\nabla u\rangle,\label{2.3}\\
\frac{1}{2}\Delta_{V}|\nabla u|^{2}
&=&|\nabla^{2}u|^{2}+{\rm Ric}^{n,m}_{V}(\nabla u,\nabla u)
+\langle\nabla\Delta_{V} u,\nabla u\rangle
+\frac{\langle V,\nabla u\rangle^{2}}{n-m}.\label{2.4}
\end{eqnarray}
for any $n>m$.
\end{lemma}

\begin{proof} When $V=\nabla f$ for some smoot function $f$,
this inequality was established by many authors (e.g., \cite{Li05}). The
proof is bases on the usual Bochner-Weitzenb\"ock formula
\begin{equation}
\frac{1}{2}\Delta|\nabla u|^{2}=|\nabla^{2}u|^{2}
+{\rm Ric}(\nabla u,\nabla u)
+\langle\nabla\Delta u,\nabla u\rangle.\label{2.5}
\end{equation}
By definition, it follows that
\begin{eqnarray*}
\frac{1}{2}\Delta_{V}|\nabla u|^{2}&=&\frac{1}{2}\Delta|\nabla u|^{2}
+\frac{1}{2}\langle V,\nabla|\nabla u|^{2}\rangle\\
&=&|\nabla^{2}u|^{2}
+{\rm Ric}(\nabla u,\nabla u)
+\langle\nabla\Delta u,\nabla u\rangle+\frac{1}{2}\langle V,
\nabla|\nabla u|^{2}\rangle.
\end{eqnarray*}
The last two terms of the right-hand side becomes
\begin{eqnarray*}
\langle\nabla\Delta u,\nabla u\rangle
+\frac{1}{2}\langle V,\nabla|\nabla u|^{2}\rangle&=&
\left\langle\nabla(\Delta_{V}u-\langle V,\nabla u\rangle),\nabla u\right\rangle
+V^{i}\nabla^{i}u\nabla_{i}\nabla_{j}u\\
&=&\langle\nabla\Delta_{V}u,\nabla u\rangle
-\nabla^{i}u\nabla_{i}(V^{j}\nabla_{j}u)
+V^{i}\nabla^{j}u\nabla_{i}\nabla_{j}u\\
&=&\langle\nabla\Delta_{V}u,\nabla u\rangle
-\nabla_{i}u\nabla_{j}u\nabla^{i}V^{j}\\
&=&\langle\nabla\Delta_{V}u,\nabla u\rangle
-\nabla_{i}u\nabla_{j}u
\left(\frac{\nabla^{i}V^{j}+\nabla^{j}V^{i}}{2}\right)\\
&=&\langle\nabla\Delta_{V}u,\nabla u\rangle
-\frac{1}{2}\mathscr{L}_{V}g(\nabla u,\nabla u).
\end{eqnarray*}
Therefore
\begin{equation*}
\frac{1}{2}\Delta_{V}|\nabla u|^{2}
=|\nabla^{2}u|^{2}
+{\rm Ric}_{V}(\nabla u,\nabla u)
+\langle\nabla\Delta_{V}u,\nabla u\rangle.
\end{equation*}
This is the identity (\ref{2.1}), which implies (\ref{2.4}) and (\ref{2.3}). From the elementary inequality $m|\nabla^{2}u|^{2}\geq|\Delta u|^{2}$ we arrive at
\begin{equation*}
\frac{1}{2}\Delta_{V}|\nabla u|^{2}\geq\frac{1}{m}|\Delta u|^{2}
+{\rm Ric}^{n,m}_{V}(\nabla u,\nabla u)
+\langle\nabla\Delta_{V}u,\nabla u\rangle
+\frac{1}{n-m}\langle V,\nabla u\rangle^{2}
\end{equation*}
for any $n>m$. Using another elementary inequality
\begin{equation*}
(a-b)^{2}\geq\frac{1}{t}a^{2}-\frac{1}{t-1}b^{2}, \ \ \  t>1,
\end{equation*}
we get
\begin{eqnarray*}
\frac{1}{m}|\Delta u|^{2}&=&\frac{1}{m}\left(\Delta_{V}u-
\langle V,\nabla u\rangle\right)^{2}\\
&\geq&\frac{1}{m}
\left(\frac{1}{n/m}(\Delta_{V}u)^{2}
-\frac{1}{n/m-1}\langle V,\nabla u\rangle^{2}\right)\\
&=&\frac{1}{n}(\Delta_{V}u)^{2}-\frac{1}{n-m}\langle V,\nabla u\rangle^{2}
\end{eqnarray*}
Together those inequalities, we obtain the desired inequality (\ref{2.2}).
\end{proof}

\begin{corollary}\label{c2.2} Let $u$ be a solution of $\Delta_{V}u=0$ and $n>m$ a constant. Then
\begin{equation}
|\nabla u|\Delta_{V}|\nabla u|\geq\frac{1}{n-1}\left|\nabla(|\nabla u|)
\right|^{2}+{\rm Ric}^{n,m}_{V}(\nabla u,\nabla u).\label{2.6}
\end{equation}
\end{corollary}

\begin{proof} From the identity
\begin{equation*}
\Delta_{V}|\nabla u|^{2}=2|\nabla u|\Delta_{V}|\nabla u|
+2\left|\nabla(|\nabla u|)\right|^{2}
\end{equation*}
and the above lemma, we obtain
\begin{equation}
|\nabla u|\Delta_{V}|\nabla u|
=|\nabla^{2}u|^{2}
-\left|\nabla(|\nabla u|)\right|^{2}
+{\rm Ric}_{V}(\nabla u,\nabla u)\label{2.7}
\end{equation}
for any solution $u$ of $\Delta_{V}u=0$. Now the proof follows from the
similar argument as stated in \cite{SchoenYau94, Yau75, Li05}. For the
completeness, we present it here. Given any point $p\in\mathcal{M}$ and choose
a normal coordinate system $(x^{1},\cdots,x^{m})$ at $p$ so that $u_{i}(p)
=|\nabla u|(p)$ and $u_{i}(p)=0$ for all $2\leq i\leq m$, where $u_{i}:=
\partial u/\partial x^{i}$, etc. Then
\begin{equation*}
\left|\nabla(|\nabla u|)\right|^{2}=\sum_{1\leq j\leq m}u^{2}_{1j}.
\end{equation*}
Since $0=\Delta u+\langle V,u\rangle$ it follows that
\begin{equation*}
-\sum_{2\leq i\leq m}u_{ii}=u_{11}+V_{1}u_{1}
\end{equation*}
and then, for any $\alpha>0$, (see page 1310--1311 in \cite{Li05} for some detail)
\begin{eqnarray*}
|\nabla^{2}u|^{2}-\left|\nabla(|\nabla u|)\right|^{2}
&\geq&\sum_{2\leq i\leq  m}u^{2}_{i1}
+\frac{1}{m-1}(u_{11}+V_{1}u_{1})^{2}\\
&\geq&\left(\sum_{2\leq i\leq m}u^{2}_{i1}
+\frac{1}{(1+\alpha)(m-1)}u^{2}_{11}\right)
-\frac{1}{\alpha(m-1)}|V_{1}u_{1}|^{2}\\
&\geq&\frac{1}{(1+\alpha)(m-1)}
\left|\nabla(|\nabla u|)\right|^{2}
-\frac{1}{\alpha(m-1)}|\langle V,\nabla u\rangle|^{2}.
\end{eqnarray*}
Consequently,
\begin{equation*}
|\nabla u|\Delta_{V}|\nabla u|
\geq\frac{1}{(1+\alpha)(m-1)}\left|\nabla(|\nabla u|)\right|^{2}
+\left({\rm Ric}_{V}
-\frac{1}{\alpha(m-1)}
V\otimes V\right)(\nabla u,\nabla u).
\end{equation*}
Taking $\alpha=\frac{n-m}{m-1}$ yields the desired result.
\end{proof}

\begin{theorem}\label{t2.3} Let $(\mathcal{M},g)$ be a compact $m$-dimensional 
Riemannian manifold with ${\rm Ric}^{n,m}_{V}\geq-K$, where $K\geq0$ is a 
constant. If $u$ is a solution of $\Delta_{V}u=0$ which is bounded from below, then
\begin{equation}
|\nabla u|\leq \sqrt{(n-1)K}\left(u-\inf_{\mathcal{M}}u\right).\label{2.8}
\end{equation}
In particular, if ${\rm Ric}^{n,m}_{V}\geq0$, then every positive solution of
$\Delta_{V}u=0$ must be constant.
\end{theorem}

\begin{proof} By replacing $u$ by $u-\inf_{\mathcal{M}}u$, we may assume that
$u$ is positive. The proof is similar to that in \cite{Yau75, SchoenYau94,
Li05}. Let $\phi:=|\nabla u|/u=|\nabla\ln u|$. Then
\begin{equation*}
\nabla\phi=\frac{\nabla|\nabla u|}{u}
-\frac{|\nabla u|\nabla u}{u^{2}}.
\end{equation*}
At any point where $\nabla u\neq0$, Using
\begin{equation*}
\Delta_{V}|\nabla u|=u\Delta_{V}\phi
+2\langle\nabla\phi,\nabla u\rangle
+\phi\Delta_{V}u
=u\Delta_{V}\phi+2\langle\nabla\phi,\nabla u\rangle
\end{equation*}
we obtain
\begin{eqnarray*}
\Delta_{V}\phi&=&\frac{\Delta_{V}|\nabla u|}{u}
-\frac{2\langle\nabla\phi,\nabla u\rangle}{u}\\
&\geq&\frac{1}{u|\nabla u|}
\left(\frac{1}{n-1}\left|\nabla(|\nabla u|)\right|^{2}
-K|\nabla u|^{2}\right)-\frac{2\langle\nabla\phi,\nabla u\rangle}{u}\\
&=&\frac{1}{n-1}\frac{|\nabla(|\nabla u|)|^{2}}{u|\nabla u|}
-K\phi-\frac{2\langle\nabla\phi,\nabla u\rangle}{u}.
\end{eqnarray*}
As \cite{SchoenYau94, Li05}, we furthermore get the following inequality
\begin{equation*}
\Delta_{V}\phi\geq-K\phi-\left(2-\frac{2}{n-1}\right)
\frac{\langle\nabla\phi,\nabla u\rangle}{u}
+\frac{1}{n-1}\phi^{3}.
\end{equation*}
If $\phi$ achieves its maximum at some point $p\in\mathcal{M}$, then $\nabla
\phi=\Delta\phi=0$ at $p$ and $\Delta_{V}\phi(p)\leq0$. Plugging this into the above inequality implies $\phi(p)\leq\sqrt{(n-1)K}$ and hence $|\nabla u|\leq\sqrt{(n-1)K}u$ on $\mathcal{M}$.
\end{proof}

Using Lemma \ref{l2.1}, Bakry and Qian \cite{BakryQian00} studied the eigenvalue problem of $\Delta_{V}$.

\section{Bakry-Qian's comparison theorem}\label{section3}

If ${\rm Ric}^{n,m}_{V}\geq K$ for some constant $K$, then the elliptic
operator $\Delta_{V}$ satisfies the $CD(K,n)$ condition in the sense
of Bakry \cite{Bakry94}, see also \cite{BakryQian05, Li05}. Bakry and Qian
proved the following Laplacian comparison theorem for $\Delta_{V}$.

\begin{theorem}\label{t3.1}{\bf (Bakry-Qian, 2005)} Let $(\mathcal{M},g)$ be a complete
$m$-dimensional Riemannian manifold and ${\rm Ric}^{n,m}_{V}\geq (n-1)K$, where $K=K(d(p))$ is a
function depending on the distance function $d(p)=d(p,p_{0})$ for a
fixed point $p_{0}\in\mathcal{M}$. Let $\theta_{K}$ be the solution defined on the maximal interval $(0,\delta_{K})$ of the Riccati equation
\begin{equation}
\dot{\theta}_{K}(r)=-K(r)-\theta^{2}_{K}(r), \ \ \
\lim_{r\to0}r\theta_{K}(r)=n-1,\label{3.1}
\end{equation}
and $\delta_{K}$ is the explosion time of $\theta_{K}$ such that
\begin{equation*}
\lim_{r\to\delta_{K}-}\theta_{K}(r)=-\infty.
\end{equation*}
Then
\begin{itemize}

\item[(i)] If $\delta_{K}<\infty$, then $\mathcal{M}$ is compact and the diameter of $(\mathcal{M},g)$ is bounded from above by $\delta_{K}$.

\item[(ii)] For any $p\in\mathcal{M}\setminus{\rm cut}(p_{0})$, we have
\begin{equation}
\Delta_{V}d\leq(n-1)\theta_{K}(d).
\end{equation}

\item[(iii)] We denote by $\mu_{V}$ an invariant measure for $\Delta_{V}$, that is a solution of $\Delta^{\ast}_{V}(\mu_{V})=0$. By ellipticity, such an invariant measure has a smooth density with respect to $dV_{g}$. Then the Laplacian comparison theorem holds in the sense of distributions:
    \begin{equation}
    \int_{\mathcal{M}}d(\Delta^{\ast}_{V}\varphi)\!\ d\mu_{V}
    \leq\int_{\mathcal{M}}\varphi(m-1)\theta_{K}(d)\!\ d\mu_{V}
    \end{equation}
    for any nonnegative smooth function $\varphi$ on $\mathcal{M}$
    with compact support.

\end{itemize}

\end{theorem}

Compared with the space-form, we obtain

\begin{corollary}\label{c3.2} If $(\mathcal{M},g)$ is a complete $m$-dimensional
Riemannian manifold with ${\rm Ric}^{n,m}_{V}\geq(n-1)K$, where $K\in{\bf R}$,
and if $p\in\mathcal{M}$, then for any $x\in\mathcal{M}$ where $d(x):=d(x,p)$ is smooth, we have
\begin{equation}
\Delta_{V}d\leq\left\{\begin{array}{cc}
(n-1)\sqrt{K}\cot\left(\sqrt{K}d\right), & K>0,\\
\frac{n-1}{d}, & K=0,\\
(n-1)\sqrt{|K|}\coth\left(\sqrt{|K|}d\right), & K<0.
\end{array}\right.\label{3.4}
\end{equation}
\end{corollary}

Using $x\coth x\leq 1+x$ yields (see also \cite{BakryQian05, Qian98})

\begin{corollary}\label{c3.3} If $(\mathcal{M},g)$ is a complete $m$-dimensional Riemannian manifold with ${\rm Ric}^{n,m}_{V}\geq(n-1)K$, where $K\leq0$, then
\begin{equation}
\Delta_{V}d\leq\frac{n-1}{d}+(n-1)\sqrt{|K|}\label{3.5}
\end{equation}
in the sense of distributions. In particular, if $(\mathcal{M},g)$ is a complete
$m$-dimensional Riemannian manifold with ${\rm Ric}^{n,m}_{V}\geq0$, then
\begin{equation}
d\Delta_{V}d\leq n-1\label{3.6}
\end{equation}
in the sense of distributions.
\end{corollary}

\begin{theorem}\label{t3.4} Let $(\mathcal{M},g)$ be a complete $m$-dimensional Riemannian
manifold with ${\rm Ric}^{n,m}_{V}\geq-(n-1)K$ where $K\geq0$ is a constant. If $u$ is a positive solution of $\Delta_{V}u=0$ on $\mathcal{M}$, then
\begin{equation}
\sup_{B(x,r/2)}\frac{|\nabla u|}{u}\leq 8(n-1)\left(\frac{1}{r}+\sqrt{K}\right).
\label{3.7}
\end{equation}
\end{theorem}

\begin{proof} Recall
\begin{equation*}
\Delta_{V}\phi\geq-(n-1)K\phi-\left(2-\frac{2}{n-1}\right)
\frac{\langle\nabla\phi,\nabla u\rangle}{u}
+\frac{1}{n-1}\phi^{3}, \ \ \ \phi:=\frac{|\nabla u|}{u}.
\end{equation*}
For any $r>0$, we consider the quantity
\begin{equation*}
F(y):=(r^{2}-d^{2}(x,y))\phi(y), \ \ \ y\in B(x,r).
\end{equation*}
It is clear that
\begin{equation*}
\nabla F=-\phi\Delta(d^{2})+(r^{2}-d^{2})\nabla\phi, \ \ \
\Delta_{V}F=(r^{2}-d^{2})\Delta_{V}\phi-\phi\Delta_{V}(d^{2})
-2\langle\nabla(d^{2}),\nabla\phi\rangle.
\end{equation*}
Now the proof of the above estimate is similar to Theorem 3.1 (page 19--20) in \cite{SchoenYau94}
or Theorem 2.3 (page 1313--1314) in \cite{Li05}. Since $F=0$ on the boundary of $B(x,r)$, if $|\nabla u|\neq0$, then $F$ must achieve its maximum at some $x_{0}\in B(x,r)$. By Calabi's argument \cite{Calabi57, ChengYau75, SchoenYau94}, we may assume that $x_{0}$ is not a cut point of $x$. Then $F$ is smooth near $x_{0}$ and hence
\begin{equation*}
\Delta F\leq0=\nabla F \ \ \ \text{at} \ x_{0}.
\end{equation*}
It follows that $\Delta_{V}F(x_{0})=\Delta F(x_{0})+\langle V,\nabla F\rangle(x_{0})
\leq0$ and then
\begin{equation*}
\frac{\nabla\phi}{\phi}=\frac{\nabla(d^{2})}{r^{2}-d^{2}}, \ \ \
\frac{\Delta_{V}\phi}{\phi}
-\frac{\Delta_{V}(d^{2})}{r^{2}-d^{2}}-\frac{2\langle\nabla(d^{2}),
\nabla\phi\rangle}{\phi(r^{2}-d^{2})}\leq0 \ \ \ \text{at} \ x_{0}.
\end{equation*}
Consequently,
\begin{equation*}
\frac{\Delta_{V}\phi}{\phi}
-\frac{\Delta_{V}(d^{2})}{r^{2}-d^{2}}
-\frac{2|\nabla(d^{2})|^{2}}{(r^{2}-d^{2})^{2}}\leq0 \ \ \ \text{at} \ x_{0}.
\end{equation*}
By (3.5) we have
\begin{equation*}
\Delta_{V}(d^{2})=2d\Delta_{V}d+2|\nabla d|^{2}\leq 2+2(n-1)
(1+\sqrt{K}d)
\end{equation*}
so that, using $|\nabla(d^{2})|^{2}=4d^{2}$,
\begin{eqnarray*}
0&\geq&\frac{\Delta_{V}\phi}{\phi}-\frac{2+2(n-1)(1+\sqrt{K}d)}{r^{2}-d^{2}}
-\frac{8d^{2}}{(r^{2}-d^{2})^{2}}\\
&\geq&-(n-1)K-\left(2-\frac{2}{n-1}\right)\frac{\langle\nabla\phi,\nabla u\rangle}{\phi u}
+\frac{1}{n-1}\phi^{2}\\
&&- \ \frac{2+2(n-1)(1+\sqrt{K}d)}{r^{2}-d^{2}}
-\frac{8d^{2}}{(r^{2}-d^{2})^{2}}
\end{eqnarray*}
at $x_{0}$. On the other hand,
\begin{equation*}
\frac{\langle\nabla\phi,\nabla u\rangle}{\phi u}
=\left\langle\frac{\nabla\phi}{\phi},\frac{\nabla u}{u}\right\rangle
=\frac{\nabla(d^{2}),\nabla u\rangle}{(r^{2}-d^{2})u}
=\frac{2d\langle\nabla d,\nabla u\rangle}{(r^{2}-d^{2})u}
\leq\frac{2d}{r^{2}-d^{2}}\phi.
\end{equation*}
Therefore
\begin{equation*}
0\geq\frac{1}{n-1}F^{2}
-\frac{4(n-2)}{n-1}dF-[2+2(n-1)(1+\sqrt{K}d)](r^{2}-d^{2})
-8d^{2}-(n-1)K(r^{2}-d^{2})^{2}
\end{equation*}
at $x_{0}$. When $n=2$, the above inequality becomes
\begin{equation*}
F\leq \sqrt{Kr^{4}+(12+2\sqrt{K}r)r^{2}}
\leq\sqrt{12}r(1+\sqrt{K}r).
\end{equation*}
When $n\geq3$, we arrive at
\begin{equation*}
\frac{1}{n-1}F^{2}-\frac{4(n-2)}{n-1}rF
\leq[2+2(n-1)(1+\sqrt{K}r)]r^{2}+8r^{2}+(n-1)Kr^{4}
\end{equation*}
and hence
\begin{eqnarray*}
F(x_{0})&\leq& r\left[2(n-2)+(n-1)\sqrt{(\sqrt{K}r)^{2}
+2\sqrt{K}r+6+\frac{2(n+1)}{(n-1)^{2}}}\right]\\
&\leq&r\left[2(n-2)+(n-1)\sqrt{8}(1+\sqrt{K}r)\right]\\
&\leq&4\sqrt{2}(n-1)r(1+\sqrt{K}r).
\end{eqnarray*}
In both case, we obtain
\begin{equation*}
F\leq4\sqrt{2}(n-1)r(1+\sqrt{K}r) \ \ \ \text{on} \ B(x,r).
\end{equation*}
In particular
\begin{equation*}
\frac{3}{4}r^{2}\sup_{B(x,r/2)}\frac{|\nabla u|}{u}
\leq\sup_{B(x,r/2)}F\leq4\sqrt{2}(n-1)r(1+\sqrt{K}r)
\end{equation*}
which implies
\begin{equation*}
\sup_{B(x,r/2)}\frac{|\nabla u|}{u}
\leq\frac{16\sqrt{2}}{3}(n-1)\left(\frac{1}{r}+\sqrt{K}\right)
\leq8(n-1)\left(\frac{1}{r}+\sqrt{K}\right).
\end{equation*}
This is the desired estimate.
\end{proof}

As an immediate consequence, we have the following variants corollaries
parallel to these in \cite{SchoenYau94, Li05}.

\begin{corollary}\label{c3.5} Let $(\mathcal{M},g)$ be a complete $m$-dimensional Riemannian manifold with ${\rm Ric}^{n,m}_{V}\geq-(n-1)K$ where $K\geq0$ is a constant.
\begin{itemize}

\item[(i)] If $(\mathcal{M},g)$ is noncompact and $u$ is a positive solution of $\Delta_{V}u=0$ on $\mathcal{M}$, then
    \begin{equation}
    \sup_{\mathcal{M}}\frac{|\nabla u|}{u}\leq 8(n-1)\sqrt{K}.
    \label{3.8}
    \end{equation}

\item[(ii)] If $u$ is a solution of $\Delta_{V}u=0$ on a geodesic ball $B(x,r)$, then
\begin{equation}
\sup_{B(x,r/2)}|\nabla u|\leq 16(n-1)\left(\frac{1}{r}+\sqrt{K}\right)
\sup_{B(x,r)}|u|.\label{3.9}
\end{equation}

\item[(iii)] If $u$ is a positive solution of $\Delta_{V}u=0$ on a geodesic ball $B(x,r)$, then
\begin{equation}
\sup_{B(x,r/2)}u\leq e^{8(n-1)(1+2r\sqrt{K})}
\inf_{B(x,r/2)}u.\label{3.10}
\end{equation}
\end{itemize}
\end{corollary}

\section{A generalized diffusion operator}\label{section4}

Recall that a triple $(\mathcal{M}, g, \mu)$ is called a {\it weighted 
Riemannian manifold} (for more detail, see \cite{Grigoryan09}), if $(\mathcal{M},g)$ is a Riemannian manifold and $\mu$ is a measure on
$\mathcal{M}$ with a smooth positive density function $f$ (that is, $d\mu=fd
V_{g}$). The {\it weighted divergence} and the {\it weighted Laplace operator} are defined by
\begin{equation*}
{\rm div}_{\mu}=\frac{1}{f}{\rm div}(f\!\ ), \ \ \ \Delta_{\mu}:=
{\rm div}_{\mu}\circ\nabla
\end{equation*}
respectively, where $\nabla$ is the Levi-Civita connection of $g$. There are two
examples of $\Delta_{\mu}$:
\begin{itemize}

\item[(a)] When $V=\nabla f$, the operator $\Delta_{V}$ is exactly the weighted
Laplace operator of the weighted Riemannian manifold $(\mathcal{M}, g,\mu)$
where $\mu=e^{f}dV_{g})$. Indeed,
\begin{equation*}
\Delta_{\mu}=\frac{1}{e^{f}}{\rm div}(e^{f}\nabla\!\ )=\frac{1}{e^{f}}
\left(e^{f}\Delta+\langle\nabla e^{f},\nabla\!\ \rangle\right)=\Delta+\langle
\nabla f,\nabla\!\ \rangle=:\Delta_{f}.
\end{equation*}

\item[(b)] In \cite{Mastrolia-Rigoli10}, the authors introduced a diffusion-type operator
\begin{equation*}
L=\frac{1}{B}{\rm div}(A\nabla\!\ )
\end{equation*}
where $A, B$ are some sufficiently smooth positive functions on $\mathcal{M}$. Set
\begin{equation*}
\tilde{g}:=\frac{B}{A}g, \ \ \ d\tilde{\mu}:=B\!\ dV_{g}.
\end{equation*}
Then $L$ is the weighted Laplace operator of the weighted Riemannian manifold $(\mathcal{M},
\tilde{g},\tilde{\mu})$ since
\begin{equation*}
\widetilde{\Delta}_{\tilde{\mu}}=
{\rm div}_{\tilde{\mu}}\circ\widetilde{\nabla}
=\frac{1}{B}{\rm div}\left(B\frac{A}{B}\nabla\!\ \right)=L.
\end{equation*}

\end{itemize}
In both cases, $\Delta_{f}$ or $L$ can be viewed as the special case of
$\Delta_{V}$ on some Riemannian manifold. In this section we study the following
diffusion Poisson equation
\begin{equation}
\Delta_{V}u=F(u)\label{4.1}
\end{equation}
on a complete noncompact $m$-dimensional Riemannian manifold $\mathcal{M}$, 
where $m\geq2$. Let $B(p,r)$ denote the geodesic ball of radius $r>0$ centered at $p$ and $d(x):={\rm dist}_{g}(x,p)$.

\begin{lemma}\label{l4.1} Let ${\rm Ric}^{n,m}_{V}\geq-(n-1)K$ on $B(p,r)$, where $K\geq0$ is a
constant and $n>m$, and $u\in C^{3}(\mathcal{M})$ is a solution of $\Delta_{V}u
=F(u)$ on $\mathcal{M}$ for some $F\in C^{1}({\bf R})$. Consider the function
\begin{equation}
H(x)=[r^{2}-d^{2}(x)]^{2}|\nabla u|^{2}(x)G[u(x)]\label{4.2}
\end{equation}
where $G\in C^{2}({\bf R})$ and $G(u)>0$ on $B(p,r)$. Then
\begin{eqnarray*}
&&\Delta_{V}\ln H+\left\langle\nabla\ln H,\nabla\ln H+\frac{8d\nabla d}{r^{2}-d^{2}}
-\frac{2G'(u)}{G(u)}\nabla u\right\rangle\\
&\geq&-2(n-1)K+2F'(u)+\frac{G'(u)}{G(u)}F(u)
+\frac{2G(u)G''(u)-3G'(u)^{2}}{2G(u)^{2}}
|\nabla u|^{2}\\
&&- \ \frac{4dG'(u)}{(r^{2}-d^{2})G(u)}|\nabla u|
-\frac{4[n+(n-1)\sqrt{K}d]}{r^{2}-d^{2}}
-\frac{16d^{2}}{(r^{2}-d^{2})^{2}},
\end{eqnarray*}
and
\begin{eqnarray*}
&&\Delta_{V}\ln H+2\left\langle\nabla\ln H,
\nabla\ln H+\frac{8d\nabla d}{r^{2}-d^{2}}
-\frac{2G'(u)}{G(u)}\nabla u\right\rangle\\
&\geq&-2(n-1)K+2F'(u)
+\frac{8G(u)G''(u)-(8+n)G''(u)^{2}}{8G(u)^{2}}|\nabla u|^{2}\\
&&- \ \frac{8dG'(u)}{(r^{2}-d^{2})G(u)}|\nabla u|
-\frac{4[n+(n-1)d\sqrt{K}]}{r^{2}-d^{2}}
-\frac{24d^{2}}{(r^{2}-d^{2})^{2}}.
\end{eqnarray*}
on points where $H$ is positive.
\end{lemma}

\begin{proof} On points where $H$ is positive, we get
\begin{eqnarray*}
\nabla\ln H&=&\frac{\nabla H}{H} \ \ = \ \ \frac{G'(u)}{G(u)}
\nabla u+\frac{\nabla|\nabla u|^{2}}{|\nabla u|^{2}}
-\frac{2\nabla(d^{2})}{r^{2}-d^{2}},\\
\Delta_{V}\ln H&=&\frac{\Delta_{V}H}{H}-|\nabla\ln H|^{2}\\
&=&-2\frac{\Delta_{V}(d^{2})}{r^{2}-d^{2}}
+\frac{\Delta_{V}|\nabla u|^{2}}{|\nabla u|^{2}}
+\frac{G'(u)}{G(u)}\Delta_{V}u-2\frac{|\nabla(d^{2})|^{2}}{(r^{2}-d^{2})^{2}}\\
&&+ \ \frac{G(u)G''(u)-G'(u)^{2}}{G(u)^{2}}|\nabla u|^{2}
-\frac{|\nabla|\nabla u|^{2}|^{2}}{|\nabla u|^{4}}.
\end{eqnarray*}
By (\ref{2.3}) and Kato's inequality
\begin{equation*}
|\nabla|\nabla u|^{2}|^{2}\leq 4|\nabla u|^{2}|\nabla^{2}u|^{2},
\end{equation*}
we arrive at
\begin{equation*}
\frac{\Delta_{V}|\nabla u|^{2}}{|\nabla u|^{2}}
\geq\frac{|\nabla|\nabla u|^{2}|^{2}}{2|\nabla u|^{4}}
-2(n-1)K+2F'(u).
\end{equation*}
Using the facts $\Delta_{V}(d^{2})\leq 2+2(n-1)(1+\sqrt{K}d)$ and $|\nabla(d^{2})|^{2}
=4d^{2}$ yields
\begin{eqnarray*}
\Delta_{V}\ln H&\geq&-2(n-1)K+2F'(u)+\frac{G'(u)}{G(u)}F(u)-\frac{|\nabla|\nabla u|^{2}|^{2}}{2|\nabla u|^{4}}\\
&&+ \ \frac{G(u)G''(u)-G'(u)^{2}}{G(u)^{2}}|\nabla u|^{2}
-\frac{4[n+(n-1)d\sqrt{K}]}{r^{2}-d^{2}}
-\frac{8d^{2}}{(r^{2}-d^{2})^{2}}.
\end{eqnarray*}
On the other hand, we have
\begin{eqnarray*}
\frac{|\nabla|\nabla u|^{2}|^{2}}{2|\nabla u|^{4}}
&=&\frac{1}{2}\left(\nabla\ln H+\frac{2\nabla(d^{2})}{r^{2}-d^{2}}
-\frac{G'(u)}{G(u)}\nabla u\right)^{2}\\
&=&\frac{G'(u)^{2}}{2G(u)^{2}}|\nabla u|^{2}
+\frac{8d^{2}}{(r^{2}-d^{2})^{2}}
-\frac{4dG'(u)}{(r^{2}-d^{2})G(u)}
\langle\nabla u,\nabla d\rangle\\
&&+ \ (\nabla\ln H)^{2}
+\left\langle\nabla\ln H,\frac{8d\nabla d}{r^{2}-d^{2}}
-\frac{2G'(u)}{G(u)}\nabla u\right\rangle
\end{eqnarray*}
which implies the following inequality
\begin{eqnarray*}
&&\Delta_{V}\ln H+\left\langle\nabla\ln H,\nabla\ln H+\frac{8d\nabla d}{r^{2}-d^{2}}
-\frac{2G'(u)}{G(u)}\nabla u\right\rangle\\
&\geq&-2(n-1)K+2F'(u)+\frac{G'(u)}{G(u)}F(u)
+\frac{2G(u)G''(u)-3G'(u)^{2}}{2G(u)^{2}}
|\nabla u|^{2}\\
&&- \ \frac{4dG'(u)}{(r^{2}-d^{2})G(u)}|\nabla u|
-\frac{4[n+(n-1)\sqrt{K}d]}{r^{2}-d^{2}}
-\frac{16d^{2}}{(r^{2}-d^{2})^{2}}.
\end{eqnarray*}

Recall the formula proved in Lemma \ref{l2.1}
\begin{equation*}
\frac{1}{2}\Delta_{V}|\nabla u|^{2}
=|\nabla^{2}u|^{2}+{\rm Ric}^{n,m}_{V}(\nabla u,\nabla u)
+\langle\nabla\Delta_{V}u,\nabla u\rangle
+\frac{1}{n-m}\langle V,\nabla u\rangle^{2}.
\end{equation*}
Therefore
\begin{equation*}
\frac{\Delta_{V}|\nabla u|^{2}}{|\nabla u|^{2}}
\geq 2\frac{|\nabla^{2}u|^{2}}{|\nabla u|^{2}}
-2(n-1)K+2F'(u)
+\frac{2}{n-m}\frac{\langle V,\nabla u\rangle^{2}}{|\nabla u|^{2}}.
\end{equation*}
As in \cite{Mastrolia-Rigoli10}, we have
\begin{equation*}
\frac{|\nabla^{2}u|^{2}}{|\nabla u|^{2}}
\geq\frac{1}{m|\nabla u|^{2}}
\left(\frac{(\Delta_{V}u)^{2}}{1+\gamma}
-\frac{\langle V,\nabla u\rangle^{2}}{\gamma}\right)
\end{equation*}
for any $\gamma>0$, and hence
\begin{equation*}
\frac{\Delta_{V}|\nabla u|^{2}}{|\nabla u|^{2}}
\geq-2(n-1)K+2F'(u)
-\frac{G'(u)}{G(u)}
\Delta_{V}u
-\frac{n}{8}\frac{G'(u)^{2}}{G(u)^{2}}|\nabla u|^{2}
\end{equation*}
by taking $\gamma=\frac{n-m}{m}$. Consequently
\begin{eqnarray*}
&&\Delta_{V}\ln H+2\left\langle\nabla\ln H,
\nabla\ln H+\frac{8d\nabla d}{r^{2}-d^{2}}
-\frac{2G'(u)}{G(u)}\nabla u\right\rangle\\
&\geq&-2(n-1)K+2F'(u)
+\frac{8G(u)G''(u)-(8+n)G''(u)^{2}}{8G(u)^{2}}|\nabla u|^{2}\\
&&- \ \frac{8dG'(u)}{(r^{2}-d^{2})G(u)}|\nabla u|
-\frac{4[n+(n-1)d\sqrt{K}]}{r^{2}-d^{2}}
-\frac{24d^{2}}{(r^{2}-d^{2})^{2}}.
\end{eqnarray*}
\end{proof}

It is observed that the above lemma is similar to that in \cite{Mastrolia10}
(Lemma 1.2, page 14). As a consequence we have

\begin{theorem}\label{t4.2} Let $(\mathcal{M},g)$ be a complete $m$-dimensional Riemannian
manifold with ${\rm Ric}^{n,m}_{V}\geq-(n-1)K(1+d^{2})^{\delta/2}$, where
$K\geq0$, $\delta<4$, and $d$ denotes the distance function from a fixed
point. If $F\in C^{1}({\bf R})$ and $u\in C^{3}(\mathcal{M})$ is a global solution of
\begin{equation*}
\Delta_{V}u=F(u)
\end{equation*}
with
\begin{equation*}
|u|\leq D(1+d)^{\nu}, \ \ \ F'(u)\geq(n-1)K(1+d^{2})^{\delta/2}
\end{equation*}
on $\mathcal{M}$ for some constants $D>0$ and $0<\nu<\min\{1,1-\frac{\delta}{4}\}$, then $u$ must be constant.
\end{theorem}

\section{Li-Yau-Hamilton estimates}\label{section5}

In this section we consider the following parabolic equation
\begin{equation}
\left(\Delta_{V}-q-\partial_{t}\right)u=au\ln u\label{5.1}
\end{equation}
on $\mathcal{M}\times(0,T]$, where $a$ is a constant and $q\in C^{2}(\mathcal{M}
\times(0,T])$. When $V\equiv0$ or $V$ is gradient, this equation was considered 
in \cite{Wu10, ZhuLi13}. Suppose that $u$ is a positive solution of (\ref{5.1}) and consider
\begin{equation}
f:=\ln u.\label{5.2}
\end{equation}
Then (\ref{5.1}) can be rewritten as
\begin{equation}
\left(\Delta_{V}-\partial_{t}\right)f=-|\nabla f|^{2}+q+af.\label{5.3}
\end{equation}

\begin{lemma}\label{l5.1} Let $(\mathcal{M},g)$ be a complete $m$-dimensional Riemannian
manifold with ${\rm Ric}^{n,m}_{V}\geq-K$, where $K$ is a nonnegative function
on $\mathcal{M}$. If $f$ is a solution of (\ref{5.3}), then the quantity
\begin{equation}
F:=t(|\nabla f|^{2}-\alpha f_{t}-\alpha q-\alpha af), \ \ \ \alpha\geq1\label{5.4}
\end{equation}
satisfies
\begin{eqnarray*}
(\Delta_{V}-\partial_{t})F&\geq&-2\langle\nabla f,\nabla F\rangle
-\frac{F}{t}-2Kt|\nabla f|^{2}
+\frac{2t}{n}\left(|\nabla f|^{2}-q-f_{t}-af\right)^{2}\\
&&- \ \alpha t\Delta_{V}g-2(\alpha-1)t\langle\nabla f,
\nabla q\rangle
-2(\alpha-1)ta|\nabla f|^{2}\\
&&+ \ \alpha at\left(|\nabla f|^{2}-q-f_{t}
-af\right).
\end{eqnarray*}
\end{lemma}

\begin{proof} By the linearity, we have
\begin{equation*}
\Delta_{V}F=t\Delta_{V}|\nabla f|^{2}-\alpha t\Delta_{V}f_{t}
-\alpha t\Delta_{V}g-\alpha at\Delta_{V}f.
\end{equation*}
Using Lemma \ref{l2.1}, together with
\begin{equation*}
\Delta_{V}f=-|\nabla f|^{2}+q+f_{t}+af=-\frac{F}{t}
-(\alpha-1)(q+f_{t}+af),
\end{equation*}
we arrive at
\begin{eqnarray*}
\Delta_{V}F&\geq&\frac{2t}{n}\left(|\nabla f|^{2}-q-f_{t}-af\right)^{2}-2t\left\langle\nabla f,\nabla\left(\frac{F}{t}+(\alpha-1)(q+f_{t}+af)\right)
\right\rangle\\
&&- \ 2Kt|\nabla f|^{2}-t\alpha\left(-\frac{F}{t}
-(\alpha-1)(q+f_{t}+af)\right)_{t}
-\alpha t\Delta_{V}g-\alpha at\Delta_{V}f\\
&=&\frac{2t}{n}\left(|\nabla f|^{2}-q-f_{t}
-af\right)^{2}-2\langle\nabla f,
\nabla F\rangle-2(\alpha-1)t\langle\nabla f,\nabla f_{t}\rangle\\
&&- \ 2(\alpha-1)t\langle\nabla f,\nabla q\rangle
-2(\alpha-1)ta|\nabla f|^{2}
-2Kt|\nabla f|^{2}+\alpha F_{t}\\
&&- \ \alpha\left(|\nabla f|^{2}-\alpha f_{t}-\alpha q-\alpha af\right)
+\alpha(\alpha-1)t q_{t}
+\alpha(\alpha-1)t f_{tt}\\
&&+ \ \alpha(\alpha-1)ta f_{t}-\alpha t\Delta_{V} q-\alpha at\Delta_{V} f.
\end{eqnarray*}
On the other hand,
\begin{equation*}
F_{t}=|\nabla f|^{2}-\alpha f_{t}-\alpha q
-\alpha af+t\left(\partial_{t}|\nabla f|^{2}
-\alpha f_{tt}-\alpha q_{t}-\alpha af_{t}\right).
\end{equation*}
This implies the result.
\end{proof}

\begin{theorem}\label{t5.2} Let $(\mathcal{M},g)$ be a compact $m$-dimensional Riemannian
manifold with ${\rm Ric}^{n,m}_{V}\geq0$. Suppose that the boundary $\partial\mathcal{M}$ of $\mathcal{M}$ is convex whenever $\partial\mathcal{M}\neq\emptyset$. Let $u$ be a positive solution of
\begin{equation*}
\left(\Delta_{V}-\partial_{t}\right)u=au\ln u
\end{equation*}
on $\mathcal{M}\times(0,T]$ for some constant $a$, with Neumann boundary condition $\frac{\partial u}{\partial\nu}=0$ on $\partial\mathcal{M}\times(0,T]$.
\begin{itemize}

\item[(1)] If $q\leq0$ then
\begin{equation*}
\frac{|\nabla u|^{2}}{u^{2}}
-\frac{u_{t}}{u}-a\ln u\leq\frac{n}{2t}-\frac{na}{2}
\end{equation*}
on $\mathcal{M}\times(0,T]$.

\item[(2)] If $a\geq0$ then
\begin{equation*}
\frac{|\nabla u|^{2}}{u^{2}}
-\frac{u_{t}}u-a\ln u\leq\frac{n}{2t}.
\end{equation*}

\end{itemize}

\end{theorem}

\begin{proof} From Lemma \ref{l5.1} we obtain
\begin{eqnarray*}
\left(\Delta_{V}-\partial_{t}\right)
F&\geq&-2\langle\nabla f,\nabla F\rangle
-\frac{F}{t}+\frac{2t}{n}(|\nabla f|^{2}
-f_{t}-af)^{2}+at(|\nabla f|^{2}-f_{t}-af)\\
&=&-2\langle\nabla f,\nabla F\rangle-\frac{F}{t}
+\frac{2F^{2}}{nt}+aF\\
&=&-2\langle\nabla f,\nabla F\rangle
+\frac{2F}{nt}\left(F-\frac{n}{2}+\frac{ant}{2}\right)
\end{eqnarray*}
where $F=t(|\nabla f|^{2}-f_{t}-af)$ and $f=\ln u$.

Now the proof is similar to that in \cite{Li-Yau86, ZhuLi13}. For convenience, we
give some detail here. Firstly we assume $a\leq0$. In this case we claim
that $F\leq\frac{n}{2}-\frac{ant}{2}$ on $\mathcal{M}
\times(0,T]$. Otherwise
\begin{equation*}
F(x_{0},t_{0})=\sup_{\mathcal{M}\times(0,T]}
F>\frac{n}{2}-\frac{ant}{2}\geq\frac{n}{2}>0
\end{equation*}
for some point $(x_{0},t_{0})\in\mathcal{M}\times(0,T]$, hence $t_{0}>0$. If
$x_{0}$ is an interior point of $\mathcal{M}$, then $\Delta F(x_{0},t_{0})
\leq\nabla F(x_{0},t_{0})=0\leq F_{t}(x_{0},t_{0})$. Consequently,
\begin{equation*}
\Delta_{V}F(x_{0},t_{0})=\Delta F(x_{0},t_{0})+\langle V,\nabla F\rangle(x_{0},
t_{0})\leq0.
\end{equation*}
At the point $(x_{0},t_{0})$ we get
\begin{equation*}
0\geq\frac{2F}{nt}\left(F-\frac{n}{2}+\frac{ant}{2}\right)
\end{equation*}
from which $F(x_{0},t_{0})\leq\frac{n}{2}-\frac{ant_{0}}{2}$. This contradiction
implies that $F\leq\frac{n}{2}-\frac{ant}{2}$ on $\mathcal{M}\times(0,T]$. Next we consider the case that $x_{0}$ is on the boundary of $\mathcal{M}$. The strong maximum principle shows
that $\frac{\partial F}{\partial\nu}(x_{0},t_{0})>0$. Choose an orthonormal
basis $(e_{i})_{1\leq i\leq m}$ for $T\mathcal{M}$, where
$e_{m}:=\partial/\partial\nu$. Compute
\begin{equation*}
F_{\nu}=2t\sum_{1\leq j\leq m-1}f_{j}f_{j\nu}
+2tf_{\nu}f_{\nu\nu}
-f_{t\nu}-af_{\nu}.
\end{equation*}
Since $u_{\nu}=0$ on $\partial\mathcal{M}$, it follows that $f_{\nu}=0$ on $\partial\mathcal{M}$ and hence
\begin{equation*}
F_{\nu}=2\sum_{1\leq j\leq m-1}f_{j}f_{j\nu}
=-2t\sum_{1\leq j,k\leq m-1}
h_{jk}f_{j}f_{k}=-2t{\rm II}(\nabla f,\nabla f)
\end{equation*}
because $f_{j\nu}=-\sum_{1\leq k\leq m-1}h_{jk}f_{k}$, where $h_{jk}$ are components of the second fundamental form {\rm II} of $\partial\mathcal{M}$. Consequently ${\rm II}
(\nabla f,\nabla f)(x_{0},t_{0})<0$ which contradicts the convexity
of $\partial\mathcal{M}$. Hence $F\leq\frac{n}{2}-\frac{ant}{2}$.

We now consider the rest case $a\geq0$.
Since $n/2t>0$, we may assume that $F\geq0$. In this case we obtain
\begin{equation*}
\left(\Delta_{V}-\partial_{t}\right)
F\geq-2\langle\nabla f,\nabla F\rangle
+\frac{2F}{nt}\left(F-\frac{n}{2}\right)
\end{equation*}
which reduces to \cite{Li-Yau86} and by the same computation we can conclude
that $F\leq n/2$.
\end{proof}

\begin{theorem}\label{t5.3} Let $(\mathcal{M},g)$ be a complete manifold with boundary $\partial\mathcal{M}$. Assume that $p\in\mathcal{M}$ and the geodesic ball $B(p,2R)$ does not intersect
$\partial\mathcal{M}$. We denote by $-K(2R)$ with $K(2R)\geq0$, a lower bound
of ${\rm Ric}^{n,m}_{V}$ on the ball $B(p,2R)$. Let $q$ be a function defined on $\mathcal{M}\times[0,T]$ which is $C^{2}$ in the $x$ variable and $C^{1}$ in the $t$ variable. Assume that
\begin{equation*}
\Delta_{V}q\leq\theta(2R), \ \ \ |\nabla q|\leq\gamma(2R)
\end{equation*}
on $B(p,2R)\times[0,T]$ for some constants $\theta(2R)$ and $\gamma(2R)$. If $u$ is a positive solution of the equation
\begin{equation*}
\left(\Delta_{V}-q-\partial_{t}\right)u=au\ln u
\end{equation*}
on $\mathcal{M}\times(0,T]$ for some constant $a$, then for any $\alpha>1$ and
$\epsilon\in(0,1)$, on $B(p,R)$, $u$ satisfies the following estimates:
\begin{itemize}

\item[(1)] for $a\geq0$, we have
\begin{eqnarray*}
|\nabla f|^{2}-\alpha f_{t}-\alpha q-\alpha af&\leq&\frac{n\alpha^{2}}{2(1-\epsilon)t}
+\frac{(A+\gamma)n\alpha^{2}}{2(1-\epsilon)}+\frac{n^{2}\beta^{4}C^{2}_{1}}{4\epsilon
(1-\epsilon)(\beta-1)R^{2}}\\
&&+ \ \frac{n\alpha^{2}[K+a(\alpha-1)]}{(1-\epsilon)(\alpha-1)}
+\left(\frac{[\alpha\theta+(\alpha-1)\gamma]n\alpha^{2}}{2(1-\epsilon)}
\right)^{1/2}.
\end{eqnarray*}

\item[(2)] for $a\leq0$, we have
\begin{eqnarray*}
|\nabla f|^{2}-\alpha f_{t}-\alpha q-\alpha af&\leq&\frac{n\alpha^{2}}{2(1-\epsilon)t}
+\frac{(A+\gamma)n\alpha^{2}}{2(1-\epsilon)}+\frac{n^{2}\beta^{4}C^{2}_{1}}{4\epsilon
(1-\epsilon)(\beta-1)R^{2}}\\
&&+ \ \frac{n\alpha^{2}[K-\frac{a}{2}a(\alpha-1)]}{(1-\epsilon)(\alpha-1)}
+\left(\frac{[\alpha\theta+(\alpha-1)\gamma]n\alpha^{2}}{2(1-\epsilon)}
\right)^{1/2}.
\end{eqnarray*}

\end{itemize}
Here $f:=\ln u$ and $A=[2C^{2}_{1}+(n-1)C^{2}_{1}(1+R\sqrt{K})
+C_{2}]/R^{2}$ for some positive constants $C_{1}, C_{2}$.
\end{theorem}

\begin{proof} Set $F:=t(|\nabla f|^{2}-\alpha f_{t}
-\alpha q-\alpha af)$. As in \cite{ChenChen09, Li-Yau86, Ma06, SchoenYau94,
ZhuLi13}, we choose a smooth function $\tilde{\varphi}(r)$ defined on $[0,\infty)$ such that
\begin{equation*}
\tilde{\varphi}(r)=\left\{\begin{array}{cc}
1, & r\in[0,1],\\
0, & r\in[2,\infty),
\end{array}\right.
\end{equation*}
and
\begin{equation*}
-C_{1}\leq\tilde{\varphi}'(r)\varphi^{-1/2}(r)\leq0, \ \ \
\tilde{\varphi}(r)\geq-C_{2}
\end{equation*}
for some positive constants $C_{1}, C_{2}$. Set
\begin{equation*}
\varphi(x):=\tilde{\varphi}\left(\frac{1}{R}d(x)\right)
\end{equation*}
where $r(x)$ denotes the distance function from $p$ to $x$. By Calabi's trick
(see, e.g., \cite{Calabi57, ChengYau75, SchoenYau94}, we may assume that the function $\varphi$ is smooth in the ball $B(p,2R)$. By Corollary 3.3, we obtain
\begin{equation*}
\frac{|\nabla\varphi|^{2}}{\varphi}\leq\frac{C^{2}_{1}}{R^{2}}, \ \ \
\Delta_{V}\varphi\geq-\frac{(n-1)C_{1}(1+R\sqrt{K})+C_{2}}{R^{2}}.
\end{equation*}
Now the proof is similar to that in \cite{ZhuLi13}; we present the detail here
for completeness. From Lemma 5.1, we arrive at
\begin{eqnarray*}
\Delta_{V}(\varphi F)&=&F\Delta_{V}\varphi+2\langle\nabla\varphi,\nabla F\rangle
+\varphi\Delta_{V}F\\
&\geq&-F\left[\frac{2C^{2}_{1}+(n-1)C_{1}(1+R\sqrt{K})+C_{2}}{R^{2}}\right]
+\frac{2}{\varphi}\langle\nabla\varphi,\nabla(\varphi F)\rangle\\
&&+ \ \varphi\bigg[F_{t}-2\langle\nabla f,\nabla F\rangle
-\frac{F}{t}-2Kt|\nabla f|^{2}+\frac{2t}{n}\left(|\nabla f|^{2}
-f_{t}-q-af\right)^{2}\\
&&- \ \alpha t\Delta_{V}q-2(\alpha-1)t\langle\nabla f,\nabla q\rangle
-2(\alpha-1)ta|\nabla f|^{2}\\
&&+ \ \alpha at\left(|\nabla f|^{2}
-f_{t}-q-af\right)\bigg]
\end{eqnarray*}
Fix a time $T'leq T$ and consider a point $(x_{0},t_{0})\in\mathcal{M}
\times[0,T']$ where $\varphi F$ achieves its maximum. Without loss of
generality, we may assume that $(\varphi F)(x_{0},t_{0})>0$ (so that
$t_{0}>0$), otherwise it is clear. Since
\begin{equation*}
\Delta(\varphi F)(x_{0},t_{0})\leq 0=\nabla(\varphi F)(x_{0},t_{0})
\leq(\varphi F)_{t}(x_{0},t_{0}),
\end{equation*}
it follows that
\begin{equation*}
\Delta_{V}(\varphi F)(x_{0},t_{0})=\Delta(\varphi F)(x_{0},t_{0})
+\langle V,\nabla(\varphi F)\rangle(x_{0},t_{0})\leq0.
\end{equation*}
Letting
\begin{equation*}
A:=\frac{2C^{2}_{1}+(n-1)C_{1}(1+R\sqrt{K})+C_{2}}{R^{2}}
\end{equation*}
and noting that $\varphi\nabla F=-F\nabla\varphi$ at the point $(x_{0},t_{0})$, we obtain
\begin{eqnarray*}
0&\geq&-AF+2F\langle\nabla f,\nabla\varphi\rangle
-\frac{\varphi F}{t_{0}}-2Kt_{0}\varphi|\nabla f|^{2}+\frac{2t_{0}}{n}\varphi
\left(|\nabla f|^{2}-f_{t}-q-af\right)^{2}\\
&&- \ \alpha t_{0}\varphi\Delta_{V}q
-2(\alpha-1)t_{0}\varphi\langle\nabla f,\nabla q\rangle
-2(\alpha-1)t_{0}a\varphi|\nabla f|^{2}\\
&&+ \ \alpha at_{0}\varphi\left(|\nabla f|^{2}-f_{t}-q-af\right)
\end{eqnarray*}
at the point $(x_{0},t_{0})$. As in \cite{ChenChen09, Yang08, ZhuLi13}, set
\begin{equation*}
\mu:=\frac{|\nabla f|^{2}}{F}(x_{0},t_{0})\geq0.
\end{equation*}
Then
\begin{eqnarray*}
|\nabla f|^{2}-f_{t}-q-af&=&F\mu+\frac{1}{\alpha}\left(\frac{F}{t_{0}}
-|\nabla f|^{2}\right)\\
&=&F\mu+\frac{F}{t_{0}\alpha}-\frac{\mu F}{\alpha} \ \ = \ \
F\left(\mu-\frac{\mu t_{0}-1}{\alpha t_{0}}\right)\\
\langle\nabla f,\nabla\varphi\rangle&\leq&|\nabla f||\nabla\varphi| \ \
\leq \ \ \frac{C_{1}}{R}\varphi^{1/2}|\nabla f|
\end{eqnarray*}
at the point $(x_{0},t_{0})$. Setting $G:=\varphi F$ and using the
above inequalities we arrive at
\begin{eqnarray*}
A t_{0}G&\geq&-\frac{2C_{1}t_{0}}{R}\mu^{1/2}G^{3/2}
-\varphi G+\frac{2}{n\alpha^{2}}[1+(\alpha-1)\mu t_{0}]^{2}G^{2}\\
&&- \ 2\varphi t^{2}_{0}[K+a(\alpha-1)]\mu G
+a\varphi t_{0}[1+(\alpha-1)\mu t_{0}]G\\
&&- \ \alpha(\varphi t_{0})^{2}\theta
-2(\alpha-1)t^{2}_{0}\varphi^{3/2}\gamma\mu^{1/2}G^{1/2}
\end{eqnarray*}
at the point $(x_{0},t_{0})$. For any $\epsilon\in(0,1)$, we have the following
elementary inequality
\begin{equation*}
\frac{2C_{1}t_{0}}{R}\mu^{1/2}G^{3/2}
\leq\frac{2\epsilon}{n\alpha^{2}}
[1+(\alpha-1)\mu t_{0}]^{2}G^{2}
+\frac{n\alpha^{2}C^{2}_{1}t^{2}_{0}\mu G}{2\epsilon R^{2}[1+(\alpha-1)\mu t_{0}]^{2}},
\end{equation*}
which, together with $2\mu^{1/2}G^{1/2}\leq 1+\mu G$, implies that
\begin{eqnarray*}
\frac{2(1-\epsilon)[1+(\alpha-1)\mu t_{0}]^{2}G^{2}}{n\alpha^{2}}
&\leq&\bigg[A t_{0}+\varphi+\frac{n\alpha^{2}C^{2}_{1}t^{2}_{0}\mu}{
2\epsilon R^{2}[1+(\alpha-1)\mu t_{0}]^{2}}\\
&&+ \ 2\varphi t^{2}_{0}[K+a(\alpha-1)]\mu
-a\varphi t_{0}[1+(\alpha-1)\mu t_{0}]\\
&&+ \ (\alpha-1)t^{2}_{0}\varphi^{3/2}\gamma\mu\bigg]G\\
&&+ \ [\alpha\varphi^{2}\theta
+(\alpha-1)\varphi^{3/2}\gamma]t^{2}_{0}
\end{eqnarray*}
at the point $(x_{0},t_{0})$. Note that $0\leq\varphi\leq1$ and $1+(\alpha-1)
\mu t_{0}\geq1$. Therefore the above inequality reduces to the following
\begin{eqnarray*}
\frac{2(1-\epsilon)G^{2}}{n\alpha^{2}}
&\leq&\bigg[At_{0}+1+\frac{n\alpha^{2}C^{2}_{1}t_{0}}{2\epsilon R^{2}(\alpha-1)}
+\frac{2\varphi t_{0}[K+a(\alpha-1)]\mu t_{0}}{[1+(\alpha-1)\mu t_{0}]^{2}}\\
&&- \ \frac{a\varphi t_{0}}{1+(\alpha-1)\mu t_{0}}
+\gamma t_{0}\bigg]G+[\alpha\theta+(\alpha-1)\gamma]t^{2}_{0}
\end{eqnarray*}
at the point $(x_{0},t_{0})$. Now the desired result follows by using the fact that
\begin{equation*}
x\leq\frac{aq}{2}+\sqrt{b+\left(\frac{a}{2}\right)^{2}}
\leq\frac{a}{2}+\sqrt{b}+\frac{a}{2}=a+\sqrt{b}
\end{equation*}
whenever $x^{2}\leq ax+b$ for some $a,b,x\geq0$. For example, when
$a\leq0$, we obtain
\begin{eqnarray*}
G^{2}&\leq&\bigg[\frac{An\alpha^{2}t_{0}}{2(1-\epsilon)}
+\frac{n\alpha^{2}}{2(1-\epsilon)}
+\frac{n^{2}\alpha^{4}C^{2}_{1}t_{0}}{4\epsilon(1-\epsilon)R^{2}
(\alpha-1)}+\frac{n\alpha^{2}[K+a(\alpha-1)]t_{0}}{(1-\epsilon)(\alpha-1)}\\
&&+ \ \frac{n\alpha^{2}\gamma t_{0}}{2(1-\epsilon)}
\bigg]G+\frac{[\alpha\theta+(\alpha-1)\gamma]n\alpha^{2}t^{2}_{0}}{2(1-
\epsilon)}
\end{eqnarray*}
at the point $(x_{0},t_{0})$, which yields an upper bound for $G$ given by
\begin{eqnarray*}
G&\leq&\left[\frac{(A+\gamma)n\alpha^{2}}{2(1-\epsilon)}
+\frac{n^{2}\alpha^{4}C^{2}_{1}}{4\epsilon(1-\epsilon)(\alpha-1)
R^{2}}+\frac{n\alpha^{2}[K+a(\alpha-1)]}{(1-\epsilon)(\alpha-1)}
\right]T'\\
&&+ \ \left(\frac{[\alpha\theta+(\alpha-1)\gamma]n\alpha^{2}}{2(1-
\epsilon)}\right)^{1/2}T'+\frac{n\alpha^{2}}{2(1-\epsilon)}
\end{eqnarray*}
at the point $(x_{0},t_{0})$. By the construction of $\varphi$, we have
$F\leq G(x_{0},t_{0})$ on $B(p,R)\times[0,T']$. Since $T'$ was arbitrary,
it proves (1). Similarly, one can get the desired result in (2).
\end{proof}

\begin{corollary}\label{c5.4} If $(\mathcal{M},g)$ is a complete noncompact
Riemannian manifold without boundary and ${\rm Ric}^{n,m}_{V}\geq-K$
on $\mathcal{M}$, then any positive solution $u$ of the equation
\begin{equation*}
\partial_{t}u=\Delta_{V}u
\end{equation*}
on $\mathcal{M}\times(0,T]$ satisfies
\begin{equation}
\frac{|\nabla u|^{2}}{u^{2}}-\alpha\frac{u_{t}}{u}
\leq\frac{n\alpha^{2}K}{\alpha-1}+\frac{n\alpha^{2}}{2t}\label{5.5}
\end{equation}
for any $\alpha>1$.
\end{corollary}

\begin{remark}\label{r5.5} As pointed in \cite{SchoenYau94}, the estimate (5.5) still
holds for any closed Riemannian manifold with ${\rm Ric}^{n,m}_{V}\geq-K$.
\end{remark}

Next we derive Hamilton's Harnack inequality for
weighted heat equation. Let $u$ be a positive solution of $\partial_{t}u=
\Delta_{V}u$.

\begin{lemma}\label{l5.6} Suppose $(\mathcal{M},g)$ is a compact Riemannian manifold. We have
\begin{equation}
\left(\partial_{t}-\Delta_{V}\right)
\frac{|\nabla u|^{2}}{u}
=-\frac{2}{u}\left[\left|\nabla^{2}u-\frac{1}{u}\nabla u
\otimes\nabla u\right|^{2}+{\rm Ric}_{V}(\nabla u,\nabla u)\right].\label{5.6}
\end{equation}
\end{lemma}

When $V\equiv0$ this identity is due to the classical result proved by Hamilton \cite{Hamilton93}. Li \cite{Li13} generalized this identity to the Witten Laplacian
$L=\Delta_{V}$, where $V=-\nabla\phi$ for some $C^{2}$-function $\phi$ on $\mathcal{M}$.

\begin{proof} As in \cite{Hamilton93, Li13}, we directly compute the evolution equation for $|\nabla u|^{2}/u$ as follows. Since $\partial_{t}u=\Delta_{V}u$, it follows that
\begin{eqnarray*}
\partial_{t}\left(\frac{|\nabla u|^{2}}{u}\right)
&=&\frac{\partial_{t}|\nabla u|^{2}}{u}-\frac{|\nabla u|^{2}}{u^{2}}
\partial_{t}u\\
&=&\frac{2}{u}\langle\nabla u,\nabla\partial_{t}u\rangle
-\frac{|\nabla u|^{2}}{u^{2}}\Delta_{V}u\\
&=&\frac{2}{u}\langle\nabla u,\nabla\Delta_{V}u\rangle-\frac{|\nabla u|^{2}}{u^{2}}
\Delta_{V}u.
\end{eqnarray*}
By the commutative formula $\nabla_{i}\Delta u=\Delta\nabla_{i}u
-R_{ij}\nabla^{j}u$ we obtain
\begin{eqnarray*}
\nabla_{i}\Delta_{V}u&=&\nabla_{i}\Delta u+\nabla_{i}(V^{j}\nabla_{j}u)\\
&=&\Delta\nabla_{i}u-R_{ij}\nabla^{j}u+V^{j}\nabla_{i}\nabla_{j}u
+\nabla_{i}V_{j}\nabla^{j}u\\
&=&\Delta_{V}\nabla_{i}u-R_{ij}\nabla^{j}u+\nabla_{i}V_{j}\nabla^{j}u.
\end{eqnarray*}
Plugging this into $\partial_{t}(|\nabla u|^{2}/u)$ yields
\begin{equation*}
\partial_{t}\left(\frac{|\nabla u|^{2}}{u}\right)
=\frac{2}{u}\left[\langle\nabla u,\Delta_{V}\nabla u\rangle
-{\rm Ric}(\nabla u,\nabla u)
+\nabla_{i}V_{j}\nabla^{i}u\nabla^{j}u\right]-\frac{|\nabla u|^{2}}{u^{2}}
\Delta_{V}u.
\end{equation*}
Because the term $\nabla^{i}V_{j}\nabla^{i}u\nabla_{j}$ is symmetric in
the indices $i,j$, we can rewrite it as
\begin{equation*}
\nabla_{i}V_{j}\nabla^{i}u\nabla^{j}u=\frac{1}{2}(\nabla_{i}V_{j}
+\nabla_{j}V_{i})\nabla^{i}u\nabla^{j}u=\frac{1}{2}\mathscr{L}_{V}g(\nabla u,
\nabla u).
\end{equation*}
Consequently
\begin{equation*}
\partial_{t}\left(\frac{|\nabla u|^{2}}{u}\right)
=\frac{2}{u}\langle\nabla u,\Delta_{V}\nabla u\rangle-\frac{2}{u}
{\rm Ric}_{V}(\nabla u,\nabla u)-\frac{|\nabla u|^{2}}{u^{2}}\Delta_{V}u.
\end{equation*}
Similarly, we compute
\begin{equation*}
\Delta_{V}\left(\frac{|\nabla u|^{2}}{u}\right)
=\frac{\Delta_{V}|\nabla u|^{2}}{u}+|\nabla u|^{2}\Delta_{V}(u^{-1})
+2\langle\nabla(u^{-1}),\nabla|\nabla u|^{2}\rangle.
\end{equation*}
Using
\begin{eqnarray*}
\Delta_{V}(u^{-1})&=&\Delta(u^{-1})
+\langle V,\nabla(u^{-1})\rangle\\
&=&-\frac{\Delta u}{u^{2}}
-\frac{2|\nabla u|^{2}}{u^{3}}
-\frac{\langle V,\nabla u\rangle}{u^{2}}\\
&=&-\frac{1}{u^{2}}\Delta_{V}u+\frac{2|\nabla u|^{2}}{u^{3}},\\
2\langle\nabla(u^{-1}),\nabla|\nabla u|^{2}\rangle
&=&-\frac{2}{u^{2}}\langle\nabla u,\nabla|\nabla u|^{2}\rangle\\
&=&-\frac{4}{u^{2}}\nabla_{i}\nabla_{j}u
\nabla^{i}u\nabla^{j}u,
\end{eqnarray*}
we get
\begin{equation*}
\Delta_{V}\left(\frac{|\nabla u|^{2}}{u}\right)
=\frac{2}{u}\langle\nabla u,\Delta_{V}\nabla u\rangle
+\frac{2}{u}|\nabla^{2}u|^{2}-\frac{|\nabla u|^{2}}{u^{2}}
\Delta_{V}u+\frac{2|\nabla u|^{4}}{u^{3}}
-\frac{4}{u^{2}}\nabla_{i}\nabla_{j}u
\nabla^{i}u\nabla^{j}u
\end{equation*}
which, together with $\partial_{t}(|\nabla u|^{2}/u)$, implies
\begin{eqnarray*}
\left(\partial_{t}-\Delta_{V}\right)\left(\frac{|\nabla u|^{2}}{u}\right)
&=&-\frac{2}{u}{\rm Ric}_{V}(\nabla u,\nabla u)\\
&&- \ \frac{2}{u}\left[|\nabla^{2}u|^{2}
+\frac{|\nabla u|^{4}}{u^{2}}-\frac{2}{u}\nabla_{i}\nabla_{j}u
\nabla^{i}u\nabla^{j}u\right].
\end{eqnarray*}
Squaring the last term on the right-hand side we obtain the desired identity.
\end{proof}

\begin{theorem}\label{t5.7} Suppose that $(\mathcal{M},g)$ is a compact Riemannian manifold with ${\rm Ric}_{V}\geq-K$ where $K\geq0$. If $u$ is a
solution of $\partial_{t}u=\Delta_{V}u$ with $0<u\leq A$ on $\mathcal{M}\times(0,T]$, then
\begin{equation}
\frac{|\nabla u|^{2}}{u^{2}}\leq
\left(\frac{2K}{e^{2Kt}-1}+2K\right)
\ln\frac{A}{u}\leq\left(\frac{1}{t}+2K\right)\ln\frac{A}{u}\label{5.7}
\end{equation}
on $\mathcal{M}\times(0,T]$.
\end{theorem}

\begin{proof} It follows from the above lemma that
\begin{equation*}
\left(\partial_{t}-\Delta_{V}\right)\left(\frac{|\nabla u|^{2}}{u}\right)
\leq\frac{2K}{u}|\nabla u|^{2}.
\end{equation*}
On the other hand, we claim that
\begin{equation*}
\left(\partial_{t}-\Delta_{V}\right)\left(u\ln\frac{A}{u}\right)
=\frac{|\nabla u|^{2}}{u}.
\end{equation*}
In fact,
\begin{eqnarray*}
\partial_{t}\left(u\ln\frac{A}{u}\right)
&=&\ln\frac{A}{u}\partial_{t}u-u\frac{\partial_{t}u}{u} \ \ = \ \ \Delta_{V}u\left(\ln\frac{A}{u}-1\right),\\
\Delta_{V}\left(u\ln\frac{A}{u}\right)&=&
\ln\frac{A}{u}\Delta_{V}u+u\Delta_{V}(\ln A-\ln u)
+2\left\langle\nabla u,-\frac{\nabla u}{u}\right\rangle\\
&=&\ln\frac{A}{u}\Delta_{V}u-u\Delta_{V}\ln u
-2\frac{|\nabla u|^{2}}{u}\\
&=&\ln\frac{A}{u}\Delta_{V}u-u\left(
\frac{\Delta_{V}u}{u}
-\frac{|\nabla u|^{2}}{u^{2}}\right)
-2\frac{|\nabla u|^{2}}{u}\\
&=&\Delta_{V}u\left(\ln\frac{A}{u}-1\right)-\frac{|\nabla u|^{2}}{u}.
\end{eqnarray*}
Choose a time-depending function $\varphi$ with $\varphi(0)=0$ and consider
\begin{equation*}
F:=\varphi\frac{|\nabla u|^{2}}{u}-u\ln\frac{A}{u}.
\end{equation*}
Therefore $F$ satisfies the following inequality
\begin{equation*}
\left(\partial_{t}-\Delta_{V}\right)F\leq(\varphi'+2K\varphi
-1)\frac{|\nabla u|^{2}}{u}.
\end{equation*}
If $\varphi$ is chosen so that $\varphi'+2K\varphi-1\leq0$, then $\partial_{t}F\leq\Delta_{V}F$ on $\mathcal{M}\times(0,T]$. By a maximum principle (e.g., see
Theorem 4.2 in \cite{ChowKnopf04}), $F\leq0$ on $\mathcal{M}\times(0,T]$ because $F(x,0)\leq0$ for all $x\in\mathcal{M}$. Solving the evolution inequality of $\varphi$ we see that
\begin{equation*}
\varphi(t)\leq\frac{1-e^{-2Kt}}{2K}=\frac{e^{2Kt}-1}{2Ke^{2Kt}}.
\end{equation*}
Since $e^{2Kt}\geq 1+2Kt$, it follows that $\frac{t}{1+2Kt}\leq\frac{e^{2Kt}-1}{2K
e^{2Kt}}$. Hence we may choose $\varphi(t)=\frac{t}{1+2Kt}$.
\end{proof}

As a consequence of Theorem \ref{t5.7}, we generalize a result in \cite{Brighton13, Li13} about the Liouville theorem.

\begin{corollary}\label{c5.8} Suppose that $(\mathcal{M},g)$ is a compact Riemannian
manifold with ${\rm Ric}_{V}\geq-K$ where $K\geq0$. If $u$ is a positive solution
of $\Delta_{V}u=0$ on $\mathcal{M}$ then
\begin{equation}
|\nabla\ln u|^{2}
\leq2K\ln\frac{\sup_{\mathcal{M}}u}{u}\label{5.8}
\end{equation}
In particular if ${\rm Ric}_{V}\geq0$ every bounded solution $u$ satisfying $\Delta_{V}u=0$ must be constant.
\end{corollary}

\begin{proof} For any $x\in\mathcal{M}$ and $t>0$, consider the function
$u(x,t):=u(x)$. Then $\partial_{t}u=\Delta_{V}u$. From (\ref{5.7}) we obtain
\begin{equation*}
|\nabla\ln u|^{2}\leq\left(\frac{2K}{e^{2Kt}-1}+2K\right)
\ln\frac{\sup_{\mathcal{M}}u}{u};
\end{equation*}
letting $t\to\infty$ implies that
$|\nabla \ln u|^{2}\leq 2K\ln(\sup_{\mathcal{M}}u/u)$.

In general, let $u$ be any bounded solution of $\Delta_{V}u=0$. For any given positive number $\epsilon>0$, replacing $u$ by $u-\inf_{\mathcal{M}}u+\epsilon$
in (\ref{5.8}) we arrive at
\begin{equation*}
\left|\nabla\ln\left(u-\inf_{\mathcal{M}}u+\epsilon\right)\right|^{2}
\leq2K\ln\frac{\sup_{\mathcal{M}}u-\inf_{\mathcal{M}}u+\epsilon}{u
-\inf_{\mathcal{M}}u+\epsilon}.
\end{equation*}
When $K=0$, this inequality shows that $|\nabla\ln(u-\inf_{\mathcal{M}}u
+\epsilon)|^{2}=0$ on $\mathcal{M}$ which means that $u-\inf_{\mathcal{M}}
u+\epsilon$
is a constant $C_{\epsilon}$. Thus $u$ must be $\inf_{\mathcal{M}}u$ a constant.
\end{proof}

Setting $V\equiv0$ in Theorem \ref{t5.7}, we obtain the classical result of
Hamilton \cite{Hamilton93}. Later Kotschwar \cite{Kotschwar07} extended Hamilton's gradient estimate to complete noncompact Riemannian manifold. Li \cite{Li13} proved Hamilton's
gradient estimate for $\Delta_{V}$ where $V=-\nabla\phi$, both in
compact case
and noncompact case. A local version of Hamilton's estimate was proved
by Souplet and Zhang \cite{SoupletZhang06} for $\Delta$, while by Arnaudon, Thalmaier, and Wang \cite{ATW09} for the general operator $\Delta_{V}$. A probabilistic proof
of Hamilton's estimates for $\Delta$ and $\Delta_{V}$ with $V=-\nabla\phi$
can be found in \cite{AT10, Li13}. In this paper we give a geometric proof
of Hamilton's estimate for Witten's Laplacian, following the method in \cite{Kotschwar07} together with Karp-Li-Grigor'yan maximum principle for complete manifolds. In an unpublished
paper \cite{KarpLi83}, Karp and Li established a maximum principle for complete
manifolds (see also \cite{Kotschwar07, LiP12, NiTam04}), which was
independently found by Grigor'yan \cite{Grigoryan87} with a slightly
weaker condition. Actually, Grigor'yan proved this type of maximum principle
for complete weighted manifolds \cite{Grigoryan87, Grigoryan09}.

\begin{theorem}\label{t5.9} {\bf (Karp-Li-Grigor'yan)} Let $(\mathcal{M},g,e^{f}dV)$ be a complete weighted manifold, and let $u(x,t)$ be a solution of
\begin{equation*}
\partial_{t}u\leq\Delta_{f}u \ \ \text{in} \ \mathcal{M}\times(0,T], \ \ \
u(\cdot,0)\leq0.
\end{equation*}
Assume that for some $x_{0}\in\mathcal{M}$ and for all $r>0$,
\begin{equation*}
\int^{T}_{0}\int_{B(x_{0},r)}
u^{2}_{+}(x,t)e^{f(x)}dV(x)dt\leq e^{\alpha(r)}
\end{equation*}
where $u_{+}:=\max\{u,0\}$ and $\alpha(r)$ is a positive increasing function on $(0,\infty)$ such that
\begin{equation*}
\int^{\infty}_{0}\frac{r}{\alpha(r)}dr=\infty.
\end{equation*}
Then $u\leq0$ on $\mathcal{M}\times(0,T]$.
\end{theorem}

The proof can be found in \cite{Grigoryan09}, Theorem 11.9, where the author
proved the result for $\partial_{u}=\Delta_{f}u$ with $u(\cdot,0)=0$, however,
the proof still works for the above setting without any changes.

\begin{theorem}\label{t5.10} Suppose that $(\mathcal{M},g)$ is a complete noncompact Riemannian manifold with ${\rm Ric}^{n,m}_{f}\geq-K$ where $K\geq0$. If $u$ is a solution of $\partial_{t}u=
\Delta_{f}u$ with $0<u\leq A$ on $\mathcal{M}\times(0,T]$, then
\begin{equation}
\frac{|\nabla u|^{2}}{u^{2}}\leq\left(\frac{2K}{e^{2Kt}-1}+2K\right)
\ln\frac{A}{u}\leq\left(\frac{1}{t}+2K\right)\ln\frac{A}{u}\label{5.9}
\end{equation}
on $\mathcal{M}\times(0,T]$.
\end{theorem}

\begin{proof} We follow the method in \cite{Kotschwar07}. Given any positive
number $\epsilon>0$, consider $u_{\epsilon}:=u+\epsilon$ and
\begin{equation*}
F_{\epsilon}:=\varphi\frac{|\nabla u_{\epsilon}|^{2}}{u_{\epsilon}}
-u_{\epsilon}\ln\frac{A_{\epsilon}}{u_{\epsilon}}
\end{equation*}
where $A_{\epsilon}:=A+\epsilon$, $\varphi(0)=0$, and $\varphi'
+2K\varphi-1\leq0$. Since $\partial_{t}u_{\epsilon}
=\Delta_{f}u_{\epsilon}$, it follows from the computation in Theorem \ref{t5.7} we
have
\begin{equation*}
\left(\partial_{t}-\Delta_{f}\right)F_{\epsilon}\leq0, \ \ \
F_{\epsilon}(\cdot,0)\leq0, \ \ \
(F_{\epsilon})_{+}\leq\frac{\varphi}{\epsilon}|\nabla u_{\epsilon}|^{2}.
\end{equation*}
Let us estimate
\begin{equation*}
\int^{T}_{0}\int_{B(x_{0},r)}\left(\frac{\varphi}{\epsilon}|\nabla
u_{\epsilon}|^{2}\right)^{2}
e^{f}dVdt.
\end{equation*}
As pointed out in the proof of Theorem \ref{t5.7}, we chosen $\varphi(t)=
(1-e^{-2Kt})/2K$. We need the following

\begin{proposition}\label{p5.11} Suppose that $(\mathcal{M},g)$ is a complete noncompact
Riemannian manifold with ${\rm Ric}^{n,m}_{V}\geq-K$ where $K\geq0$. If $u$ is a solution of $\partial_{t}u=\Delta_{V}u$ with $0<u\leq A$ on $\mathcal{M}\times(0,T]$, then
for any $a>2$ we have
\begin{equation}
\varphi|\nabla u|^{2}\leq\frac{(a+1)^{3}A^{2}}{2a(a-2)}
\left\{1+\left(
1-e^{-2Kt}\right)\left[\frac{1}{a}+C\frac{1+(n-1)(1+r\sqrt{K})}{2K r^{2}}\right]
\right\}\label{5.10}
\end{equation}
on $B(x_{0},r)\times[0,T]$ for some positive constant $C$. In particular,
\begin{equation}
\varphi|\nabla u|^{2}\leq\frac{(a+1)^{3}}{2a^{2}(a-2)}\left(a+1-e^{-2Kt}
\right)A^{2}\label{5.11}
\end{equation}
on $\mathcal{M}\times(0,T]$ for any $a>2$.
\end{proposition}

\begin{proof} Compute
\begin{eqnarray*}
\partial_{t}|\nabla u|^{2}&=&2\langle\nabla u,\nabla\Delta_{V}u\rangle\\
&=&2\nabla^{i}u\left(\Delta_{V}\nabla_{i}u-R_{ij}\nabla^{j}u
+\nabla_{i}V^{j}\nabla_{j}u\right)\\
&=&\Delta_{V}|\nabla u|^{2}-2|\nabla^{2}u|^{2}-2{\rm Ric}_{V}(\nabla u,
\nabla u),\\
\partial_{t}u^{2}&=&\Delta_{V}u^{2}-2|\nabla u|^{2}.
\end{eqnarray*}
Consider the quantity
\begin{equation*}
G:=(aA^{2}+u^{2})|\nabla u|^{2}, \ \ \ a>0,
\end{equation*}
which satisfies the following evolution equation
\begin{equation*}
\left(\partial_{t}-\Delta_{V}\right)G=
-2|\nabla u|^{4}-2(aA^{2}+u^{2})\left[|\nabla^{2}u|^{2}
+{\rm Ric}_{V}(\nabla u,\nabla u)\right]-8u\nabla_{i}\nabla_{j}u\nabla^{i}\nabla^{j}u.
\end{equation*}
From the Cauchy inequality, we have $8u\nabla_{i}\nabla_{j}u
\nabla^{i}u\nabla^{j}u\leq\eta|\nabla u|^{4}+\frac{16}{\eta}
u^{2}|\nabla^{2}u|^{2}$ for any $\eta>0$, and hence
\begin{eqnarray*}
\left(\partial_{t}-\Delta_{V}\right)G&\leq&
(\eta-2)|\nabla u|^{4}+\left[\frac{16}{\eta}-2(1+a)\right]
u^{2}|\nabla^{2}u|^{2}+2(1+a)KA^{2}|\nabla u|^{2}\\
&=&\frac{4-2a}{1+a}\left(\frac{G}{aA^{2}+u^{2}}\right)^{2}
+\frac{2(1+a)A^{2}}{aA^{2}+u^{2}}KG\\
&\leq&-2\frac{a-2}{(a+1)^{3}}\frac{G^{2}}{A^{4}}+2\frac{a+1}{a}KG
\end{eqnarray*}
where we chosen $\eta=\frac{8}{1+a}$ in the second step and $a>2$ in the third
step. Here we used a fact that
\begin{equation*}
{\rm Ric}_{V}(\nabla u,\nabla u)={\rm Ric}^{n,m}_{V}(\nabla u,
\nabla u)+\frac{\langle V,\nabla u\rangle^{2}}{n-m}\geq-K.
\end{equation*}
As in the proof of Theorem \ref{t5.3}, we take a smooth function $\chi$ equal to
$1$ on $B(x_{0},r)$ and supported in $B(x_{0},2r)$, satisfying
\begin{equation*}
\frac{|\nabla\chi|^{2}}{\chi}
\leq\frac{C^{2}_{1}}{r^{2}}, \ \ \
\Delta_{V}\chi\geq-\frac{(n-1)C_{1}(1+r\sqrt{K})+C_{2}}{r^{2}}
\end{equation*}
for some positive constants $C_{1}, C_{2}$. Because
\begin{eqnarray*}
\left(\partial_{t}-\Delta_{V}\right)
(\varphi\chi G)&=&\varphi'\chi G+\varphi\chi\left(\partial_{t}-\Delta_{V}\right)G
-\varphi G\Delta_{V}\chi\\
&&- \ 2\varphi\left\langle\frac{\nabla(\varphi\chi G)}{\varphi\chi}
-G\frac{\nabla\chi}{\chi},\nabla\chi\right\rangle,
\end{eqnarray*}
applying the above inequalities to $\varphi\chi G$ yields
\begin{eqnarray*}
\left(\partial_{t}-\Delta_{V}\right)(\varphi\chi G)
&\leq&\left[\varphi'\chi+2\varphi\frac{C^{2}_{1}}{r^{2}}
+\varphi\frac{(n-1)C_{1}(1+r\sqrt{K})+C_{2}}{r^{2}}\right]G\\
&&+\varphi\chi\left[-\frac{2(a-2)}{(a+1)^{3}}
\frac{G^{2}}{A^{4}}+\frac{2(a+1)}{a}KG\right]-2
\left\langle\nabla(\varphi\chi G),\frac{\nabla\chi}{\chi}\right\rangle\\
&=&-\frac{2(a-2)}{(a+1)^{3}A^{4}}\varphi\chi G^{2}-2\left\langle\nabla(\varphi\chi G),\frac{\nabla\chi}{\chi}\right\rangle\\
&&+ \ \bigg[\left(\varphi'
+\frac{2(a+1)}{a}K\varphi\right)\chi\\
&&+ \ \frac{\varphi}{r^{2}}
\left(2C^{2}_{1}+C_{2}+(n-1)C_{1}(1+r\sqrt{K})\right)\bigg]G.
\end{eqnarray*}
Let $(x_{0},t_{0})$ be a point where $\varphi\chi G$ achieves its maximum.
Then
\begin{equation*}
\varphi\chi G\leq\frac{(a+1)^{3}A^{4}}{2(a-2)}
\left[\varphi'+\frac{2(a+1)}{a}K\varphi+\frac{\varphi}{r^{2}}
\left(C_{3}+C_{1}(n-1)(1+r\sqrt{K})\right)\right]
\end{equation*}
at the point $(x_{0},t_{0})$, where $C_{3}:=2C^{2}_{1}+C_{2}$. Locating on $B(x_{0},r)
\times(0,T]$ we derive the desired inequality.
\end{proof}

Using (\ref{5.10}) we obtain
\begin{equation*}
\varphi|\nabla u_{\epsilon}|^{2}\leq C\frac{1+r+r^{2}}{r^{2}}A^{2}
\end{equation*}
for some positive constant $C$ depending only on $n,K$. Therefore
\begin{equation*}
\int^{T}_{0}\int_{B(x_{0},r)}
\left(\frac{\varphi}{\epsilon}|\nabla u_{\epsilon}|^{2}\right)^{2}e^{f}dVdt
\leq\frac{C^{2}T A^{4}}{\epsilon^{2}}\frac{(1+r+r^{2})^{2}}{r^{4}}
\int_{B(x_{0},r)}e^{f}dV=:e^{\alpha(r)}.
\end{equation*}
By the Bishop-Gromov volume comparison theorem for $\Delta_{f}$
(see \cite{Lott03}, or \cite{WeiWylie09}, Theorem 4.1), we see that $\int^{\infty}
rdr/\alpha(r)$ is infinity and hence by Karp-Li-Grigor'yan's maximum principle
we obtain $F_{\epsilon}\leq0$. Letting $\epsilon\to0$ implies (\ref{5.9}).
\end{proof}

\begin{remark}\label{r5.12} We compare other Hamilton's estimates with (\ref{5.9}). In our
geometric proof we require the curvature condition ${\rm Ric}^{n,m}_{f}\geq-K$
in order to use the Bakry-Qian's Laplacian comparison theorem without
any additional requirement on the potential function $f$. If we use the curvature condition ${\rm Ric}_{f}\geq-K$ in our geometric proof, then some conditions on $f$ would be
required (see \cite{ChenJostQiu12, WeiWylie09}). A probabilistic proof
of Li \cite{Li13} shows a similar estimate
\begin{equation*}
\frac{|\nabla u|^{2}}{u^{2}}\leq\left(\frac{2}{t}+2K\right)\ln\frac{A}{u}
\end{equation*}
where $0<u\leq A$ on $\mathcal{M}\times(0,T]$ and ${\rm Ric}_{f}\geq-K$.
\end{remark}

\section{Hessian estimates}\label{section6}

In this section we generalize Hessian estimates of the heat equation
in \cite{HanZhang12} to the $V$-heat equation.

\begin{theorem}\label{t6.1} Let $(\mathcal{M},g)$ be a closed $m$-dimensional
Riemannian manifold with ${\rm Ric}^{n,m}_{V}\geq-K$ where $K\geq0$.
\begin{itemize}

\item[(a)] If $u$ is a solution of $\partial_{t}u=\Delta_{V}u$ in $\mathcal{M}
\times(0,T]$ and $0<u\leq A$, then
\begin{equation}
\nabla^{2}u\leq\left(B+\frac{5}{t}\right)u\left(1+\ln\frac{A}{u}\right)g\label{6.1}
\end{equation}
in $\mathcal{M}\times(0,T]$, where
\begin{equation*}
B=\sqrt{16 m^{\frac{3}{2}}K_{1}\sup_{\mathcal{M}}|V|^{2}+2m K_{2}+3m KK_{2}+
14 m^{\frac{3}{2}}nKK_{1}+100 n^{2}m^{3}
(K_{1}+K_{2})^{2}}
\end{equation*}
with $K_{1}=\max_{\mathcal{M}}(|{\rm Rm}|+|{\rm Ric}_{V}|)$ and $K_{2}
=\max_{\mathcal{M}}|\nabla{\rm Ric}_{V}|$.

\item[(b)] If $u$ is a solution of $\partial_{t}u=\Delta_{V}u$ in $Q_{R,T}(x_{0},t_{0})$ and $0<u\leq A$, then
    \begin{equation}
    \nabla^{2}u\leq C_{1}\left(\frac{1}{T}+\frac{1+R\sqrt{K}}{R^{2}}+B\right)u
    \left(1+\ln\frac{A}{u}\right)^{2}g\label{6.2}
    \end{equation}
    in $Q_{R/2,T/2}(x_{0},t_{0})$, where
    \begin{equation*}
    B=C_{2}m^{5/2}n^{2}\left[K_{1}+K_{2}+\sqrt{(K_{1}+K_{2})K
    +K_{2}+K_{1}\sup_{\mathcal{M}}|V|^{2}}\right]
    \end{equation*}
    and $C_{1}, C_{2}$ are positive universal constants.

\end{itemize}

\end{theorem}

Let $(\mathcal{M},g)$ be a closed $m$-dimensional Riemannian manifold
with ${\rm Ric}^{n,m}_{V}\geq-K$ where $K\geq0$, and $u$ a solution of
\begin{equation}
\partial_{t}u=\Delta_{V}u\label{6.3}
\end{equation}
in $\mathcal{M}\times(0,T]$, where $T\in(0,\infty)$, and $0<u\leq A$. Set
\begin{equation}
f:=\ln\frac{u}{A}\label{6.4}
\end{equation}
as in \cite{HanZhang12}. Then
\begin{equation*}
\nabla f=\frac{\nabla u}{u}, \ \ \ \nabla^{2}f=\frac{\nabla^{2}u}{u}
-\frac{\nabla u\otimes\nabla u}{u^{2}}, \ \ \ \Delta f=\frac{\Delta u}{u}
-\frac{|\nabla u|^{2}}{u^{2}},
\end{equation*}
and
\begin{equation}
\partial_{t}f=\frac{\partial_{t}u}{u}=\frac{\Delta_{V}u}{u}=
\Delta_{V}f+|\nabla f|^{2}.\label{6.5}
\end{equation}
As in \cite{HanZhang12}, we introduce the following quantities
\begin{eqnarray}
v_{ij}&:=&\frac{\nabla_{i}\nabla_{j}u}{u(1-f)}, \ \ \
w_{ij} \ \ := \ \ \frac{\nabla_{i}u\nabla_{j}u}{u^{2}(1-f)^{2}}, \label{6.6}\\
V&:=&(v_{ij}), \ \ \ W \ \ := \ \ (w_{ij}), \ \ \ w \ \ := \ \ {\rm tr}(W) \ \
= \ \ \frac{|\nabla u|^{2}}{u^{2}(1-f)^{2}}.\label{6.7}
\end{eqnarray}
Using $\partial_{t}(u(1-f))=(1-f)\partial_{t}u-u\partial_{t}f=
-f\partial_{t}u$ we have
\begin{equation}
\partial_{t}v_{ij}=\frac{\nabla_{i}\nabla_{j}\partial_{t}u}{u(1-f)}
+f\frac{\partial_{t}u\nabla_{i}\nabla_{j}u}{u^{2}(1-f)^{2}}.\label{6.8}
\end{equation}
Similarly,
\begin{equation}
\nabla_{k}v_{ij}=\frac{\nabla_{k}\nabla_{i}\nabla_{j}u}{u(1-f)}
+f\frac{\nabla_{k}u\nabla_{i}\nabla_{j}u}{u^{2}(1-f)^{2}}.\label{6.9}
\end{equation}
By the commutation formula (see \cite{HanZhang12}, page 4) we have
\begin{eqnarray*}
\partial_{t}\nabla_{i}\nabla_{j}u
&=&\nabla_{i}\nabla_{j}\left(\Delta u+\langle V,\nabla u\rangle\right)\\
&=&\Delta\nabla_{i}\nabla_{j}u
+2R_{kij\ell}\nabla^{k}\nabla^{\ell}u-R_{i\ell}\nabla_{j}\nabla^{\ell}u
-R_{j\ell}\nabla_{i}\nabla^{\ell}u\\
&&- \ \left(\nabla_{i}R_{j\ell}+\nabla_{j}
R_{i\ell}-\nabla_{\ell}R_{ij}\right)\nabla^{\ell}u
+\nabla_{i}\nabla_{j}\langle V,\nabla u\rangle.
\end{eqnarray*}
The last term on the right-hand side is equal to
\begin{eqnarray*}
\nabla_{i}\nabla_{j}\langle V,\nabla u\rangle&=&\nabla_{i}\left(\nabla_{k}u
\nabla_{j}V^{k}+V^{k}\nabla_{j}\nabla_{k}u\right)\\
&=&\nabla_{k}u\nabla_{i}\nabla_{j}V^{k}+\nabla_{i}
\nabla_{k}u\nabla_{j}V^{k}+\nabla_{i}V^{k}\nabla_{j}\nabla_{k}u
+V^{k}\nabla_{i}\nabla_{j}\nabla_{k}u;
\end{eqnarray*}
using the commutation formula
\begin{equation*}
\nabla_{i}\nabla_{j}\nabla_{k}u
=\nabla_{i}\nabla_{k}\nabla_{j}u=\nabla_{k}\nabla_{i}
\nabla_{j}u-R_{ikj\ell}\nabla^{\ell}u
\end{equation*}
we arrive at
\begin{eqnarray*}
\nabla_{i}\nabla_{j}\langle V,\nabla u\rangle
&=&V^{k}\nabla_{k}\nabla_{i}\nabla_{j}u
+R_{kij\ell}V^{k}\nabla^{\ell}u+\nabla_{k}u\nabla_{i}\nabla_{j}V^{k}\\
&&+ \ \nabla_{i}\nabla_{k}u\nabla_{j}V^{k}
+\nabla_{j}\nabla_{k}u\nabla_{i}V^{k}.
\end{eqnarray*}
Therefore
\begin{eqnarray}
\partial_{t}\nabla_{i}\nabla_{j}u
&=&\Delta_{V}\nabla_{i}\nabla_{j}u+R_{kij\ell}\left(2\nabla^{k}
\nabla^{\ell}u+V^{k}\nabla^{\ell}u\right)\nonumber\\
&&- \ \left(\nabla_{i}R_{j}{}^{k}
+\nabla_{j}R_{i}{}^{k}-\nabla^{k}R_{ij}-\nabla_{i}\nabla_{j}V^{k}\right)
\nabla_{k}u\label{6.10}\\
&&- \ \nabla_{i}\nabla_{k}u
\left(R_{j}{}^{k}-\nabla_{j}V^{k}\right)
-\nabla_{j}\nabla_{k}u\left(R_{i}{}^{k}-\nabla_{i}V^{k}\right).\nonumber
\end{eqnarray}
Interchanging $i$ and $j$ in (\ref{6.10}) and then adding it into (\ref{6.10}) imply
\begin{eqnarray}
\left(\partial_{t}-\Delta_{V}\right)\nabla_{i}\nabla_{j}u&=&R_{kij\ell}
\left(2\nabla^{k}\nabla^{\ell}u+\frac{V^{k}\nabla^{\ell}u+V^{\ell}\nabla^{k}u}{2}\right)
\nonumber\\
&&- \ \left(\nabla_{i}R_{j}{}^{k}+\nabla_{j}R_{i}{}^{k}
-\nabla^{k}R_{ij}-\frac{\nabla_{i}\nabla_{j}V^{k}
+\nabla_{j}\nabla_{i}V^{k}}{2}\right)\nabla_{k}u\label{6.11}\\
&&- \ \nabla_{i}\nabla_{k}u\left(R_{j}{}^{k}-\nabla_{j}V^{k}\right)
-\nabla_{j}\nabla_{k}u\left(R_{i}{}^{k}
-\nabla_{i}V^{k}\right).\nonumber
\end{eqnarray}
Recall the Bakry-Emery Ricci curvatures
\begin{equation*}
{\rm Ric}_{V}:={\rm Ric}-\frac{1}{2}\mathscr{L}_{V}g, \ \ \
{\rm Ric}^{n,m}_{V}:={\rm Ric}_{V}-\frac{1}{n-m}V\otimes V.
\end{equation*}
Then
\begin{eqnarray*}
\nabla_{k}u\nabla_{i}({\rm Ric}_{V})_{j}{}^{k}
&=&\nabla_{k}u\nabla_{i}\left(R_{j}{}^{k}-\frac{\nabla_{j}V^{k}+\nabla^{k}V_{j}}{2}
\right)\\
&=&\nabla_{k}u\nabla_{i}\left(R_{j}{}^{k}-\frac{1}{2}\nabla_{j}V^{k}
\right)-\frac{1}{2}\nabla_{k}u\nabla_{i}\nabla^{k}V_{j},\\
\nabla_{k}u\nabla^{k}({\rm Ric}_{V})_{ij}&=&\nabla_{k}u\nabla^{k}
\left(R_{ij}-\frac{\nabla_{i}V_{j}+\nabla_{j}V_{i}}{2}\right)\\
&=&\nabla_{k}u\nabla^{k}R_{ij}-\frac{1}{2}\nabla^{k}u
\left(\nabla_{k}\nabla_{i}V_{j}
+\nabla_{k}\nabla_{j}V_{i}\right)\\
&=&\nabla_{k}u\nabla^{k}R_{ij}\\
&&- \ \frac{1}{2}\nabla^{k}u\left(\nabla_{i}\nabla_{k}V_{j}
+\nabla_{j}\nabla_{k}V_{i}-R_{kij\ell}V^{\ell}
-R_{kji\ell}V^{\ell}\right).
\end{eqnarray*}
The middle term on the right-hand side of (\ref{6.11}) can be now rewritten as
\begin{eqnarray*}
&&\left(\nabla_{i}R_{j}{}^{k}+\nabla_{j}R_{i}{}^{k}
-\nabla^{k}R_{ij}-\frac{\nabla_{i}\nabla_{j}V^{k}
+\nabla_{j}\nabla_{i}V^{k}}{2}\right)\nabla_{k}u\\
&=&\left[\nabla_{i}({\rm Ric}_{V})_{j}{}^{k}
+\nabla_{j}({\rm Ric}_{V})_{i}{}^{k}
-\nabla^{k}({\rm Ric}_{V})_{ij}\right]\nabla_{k}u+\frac{1}{2}R_{kij\ell}\left(V^{\ell}\nabla^{k}u+V^{k}
\nabla^{\ell}u\right).
\end{eqnarray*}
Therefore
\begin{eqnarray}
\left(\partial_{t}-\Delta_{V}\right)\nabla_{i}\nabla_{j}u
&=&2R_{kij\ell}\nabla^{k}\nabla^{\ell}u-\nabla_{i}\nabla_{k}u({\rm Ric}_{V})_{j}{}^{k}
-\nabla_{j}\nabla_{k}u({\rm Ric}_{V})_{i}{}^{k}\nonumber\\
&&- \ \bigg(\nabla_{i}({\rm Ric}_{V})_{j}{}^{k}
+\nabla_{j}({\rm Ric}_{V})_{i}{}^{k}
-\nabla^{k}({\rm Ric}_{V})_{ij}\bigg)\nabla_{k}u\label{6.12}\\
&&- \ \nabla_{i}\nabla^{k}u\frac{\nabla_{k}V_{j}
-\nabla_{j}V_{k}}{2}-\nabla_{j}\nabla^{k}u\frac{\nabla_{k}V_{i}
-\nabla_{i}V_{k}}{2}.\nonumber
\end{eqnarray}

\begin{lemma}\label{l6.2} We have
\begin{eqnarray*}
\left(\partial_{t}-\Delta_{V}\right)v_{ij}&=&-\frac{2f}{1-f}\nabla^{k}f\nabla_{k}v_{ij}
-\frac{|\nabla f|^{2}}{1-f}v_{ij}+\frac{1}{u(1-f)}\bigg[2R_{kij\ell}\nabla^{k}\nabla^{\ell}u\\
&&- \ \nabla_{i}\nabla_{k}u({\rm Ric}_{V})_{j}{}^{k}
-\nabla_{j}\nabla_{k}({\rm Ric}_{V})_{i}{}^{k}
-\bigg(\nabla_{i}({\rm Ric}_{V})_{j}{}^{k}\\
&&+ \ \nabla_{j}({\rm Ric}_{V})_{i}{}^{k}-\nabla^{k}({\rm Ric}_{V})_{ij}\bigg)
\nabla_{k}u-\nabla_{i}\nabla^{k}u\frac{\nabla_{k}V_{j}-\nabla_{j}V_{k}}{2}\\
&&- \ \nabla_{j}\nabla^{k}u\frac{\nabla_{k}V_{i}
-\nabla_{i}V_{k}}{2}\bigg]
\end{eqnarray*}
\end{lemma}

\begin{proof} Using $u\nabla f=\nabla u$ and an identity in \cite{HanZhang12}
(page 4, line -6) we have
\begin{eqnarray*}
\Delta_{V}v_{ij}&=&\Delta v_{ij}+V^{k}\nabla_{k}v_{ij}\\
&=&\frac{\Delta\nabla_{i}\nabla_{j}u}{u(1-f)}
+\frac{f\Delta u\nabla_{i}\nabla_{j}u}{u^{2}(1-f)^{2}}
+\frac{2f\nabla^{k}u\nabla_{k}\nabla_{i}\nabla_{j}u}{u^{2}(1-f)^{2}}
+\frac{\nabla_{i}\nabla_{j}u\langle\nabla u,\nabla f\rangle}{u^{2}(1-f)^{2}}\\
&&+ \ \frac{2f^{2}\nabla_{i}\nabla_{j}u|\nabla u|^{2}}{u^{3}(1-f)^{3}}
+\frac{V^{k}\nabla_{k}\nabla_{i}\nabla_{j}u}{u(1-f)}
+\frac{\langle V,\nabla u\rangle f\nabla_{i}\nabla_{j}u}{u^{2}(1-f)^{2}}\\
&=&\frac{\Delta_{V}\nabla_{i}\nabla_{j}u}{u(1-f)}
+\frac{f\Delta_{V} u\nabla_{i}\nabla_{j}u}{u^{2}(1-f)^{2}}
+\frac{2f\nabla^{k}u\nabla_{k}\nabla_{i}\nabla_{j}u}{u^{2}(1-f)^{2}}
+\frac{\nabla_{i}\nabla_{j}u\langle\nabla u,\nabla f\rangle}{u^{2}(1-f)^{2}}\\
&&+ \ \frac{2f^{2}\nabla_{i}\nabla_{j}u|\nabla u|^{2}}{u^{3}(1-f)^{3}}.
\end{eqnarray*}
Similarly,
\begin{eqnarray*}
\partial_{t}v_{ij}&=&\frac{\partial_{t}\nabla_{i}\nabla_{j}u}{u(1-f)}
-\frac{\nabla_{i}\nabla_{j}u}{u^{2}(1-f)^{2}}
[\partial_{t}u(1-f)-u\partial_{t}f]\\
&=&\frac{\partial_{t}\nabla_{i}\nabla_{j}u}{u(1-f)}
+\frac{f\Delta_{V}u\nabla_{i}\nabla_{j}u}{u^{2}(1-f)^{2}}.
\end{eqnarray*}
Combing these two identities yields
\begin{eqnarray*}
\left(\partial_{t}-\Delta_{V}\right)v_{ij}
&=&\frac{1}{u(1-f)}\left(\partial_{t}-\Delta_{V}\right)
\nabla_{i}\nabla_{j}u-\frac{2f\nabla^{k}f\nabla_{k}\nabla_{i}
\nabla_{j}u}{u(1-f)^{2}}
-\frac{\nabla_{i}\nabla_{j}u|\nabla f|^{2}}{u(1-f)^{2}}\\
&&- \ \frac{2\nabla_{i}\nabla_{j}u}{u(1-f)^{3}}f^{2}|\nabla f|^{2}.
\end{eqnarray*}
Using (\ref{6.9}) and (\ref{6.12}) we prove the desired identity.
\end{proof}

When $V$ is gradient (i.e., $V=\nabla \phi$ for some smooth function $\phi$
on $\mathcal{M}$), Lemma \ref{l6.2} reduces to
Lemma 2.1 in \cite{HanZhang12} where $\Delta$ is replaced by $\Delta_{\phi}$.

\begin{lemma}\label{l6.3} We have
\begin{eqnarray*}
\left(\partial_{t}-\Delta_{V}\right)w_{ij}&=&-\frac{2f}{1-f}\nabla^{k}f\nabla_{k}w_{ij}
-\frac{2|\nabla f|^{2}}{1-f}w_{ij}-2(v_{ik}+fw_{ik})(v_{j}{}^{k}
+fw_{j}{}^{k})\\
&&- \ ({\rm Ric}_{V})_{i}{}^{k}w_{jk}-({\rm Ric}_{V})_{j}{}^{k}w_{ik}\\
&&- \ w_{i}{}^{k}\frac{\nabla_{k}V_{j}-\nabla_{j}V_{k}}{2}
-w_{j}{}^{k}\frac{\nabla_{k}V_{i}
-\nabla_{i}V_{k}}{2}.
\end{eqnarray*}
\end{lemma}

\begin{proof} Compute
\begin{eqnarray*}
\partial_{t}w_{ij}&=&\frac{\nabla_{i}u\nabla_{j}
\partial_{t}u+\nabla_{j}u\nabla_{i}
\partial_{t}u}{u^{2}(1-f)^{2}}
+\frac{2f\partial_{t}u\nabla_{i}u\nabla_{j}u
}{u^{3}(1-f)^{3}},\\
\nabla_{k}w_{ij}&=&\frac{\nabla_{i}u\nabla_{j}\nabla_{k}u
+\nabla_{j}u\nabla_{i}\nabla_{k}u}{u^{2}(1-f)^{2}}
+\frac{2f\nabla_{i}u\nabla_{j}u\nabla_{k}u}{u^{3}(1-f)^{3}}.
\end{eqnarray*}
By the identity in \cite{HanZhang12} (page 5, line 14), we have
\begin{eqnarray*}
\Delta_{V}w_{ij}&=&\Delta w_{ij}+V^{k}\left(\frac{\nabla_{i}u
\nabla_{j}\nabla_{k}u}{u^{2}(1-f)^{2}}+\frac{2f\nabla_{i}\nabla_{j}u
\nabla_{k}u}{u^{3}(1-f)^{3}}\right)\\
&=&\frac{\nabla_{i}u\nabla_{j}\Delta u
+2\nabla_{i}\nabla_{k}u\nabla_{j}\nabla^{k}u+\nabla_{j}u
\nabla_{i}\Delta u}{u^{2}(1-f)^{2}}+R_{i}{}^{k}\frac{\nabla_{k}u
\nabla_{j}u}{u^{2}(1-f)^{2}}\\
&&+ \ R_{j}{}^{k}\frac{\nabla_{k}u\nabla_{i}u}{u^{2}(1-f)^{2}}
+\frac{4f\nabla^{k}u(\nabla_{i}u\nabla_{j}
\nabla_{k}u+\nabla_{j}u\nabla_{i}\nabla_{k}u)}{u^{3}(1-f)^{3}}\\
&&+ \ \frac{2\nabla_{i}u\nabla_{j}u(\langle\nabla u,\nabla f\rangle
+f\Delta u)}{u^{3}(1-f)^{3}}+\frac{6f^{2}|\nabla u|^{2}
\nabla_{i}u\nabla_{j}u}{u^{4}(1-f)^{4}}\\
&&+ \ \frac{ V^{k}\nabla_{k}\nabla_{j}u\nabla_{i}u}{u^{2}(1-f)^{2}}
+\frac{V^{k}\nabla_{k}\nabla_{i}u\nabla_{j}u}{u^{2}(1-f)^{2}}+\frac{2f\langle V,\nabla u\rangle\nabla_{i}u\nabla_{j}u}{u^{3}(1-f)^{3}}.
\end{eqnarray*}
Since $\Delta u=\Delta_{V}u-V^{k}\nabla_{k}u$, it follows that
\begin{equation*}
\nabla_{j}\Delta u=\nabla_{j}\Delta_{V}u
-\nabla_{k}u\nabla_{j}V^{k}-V^{k}\nabla_{k}\nabla_{j}u
\end{equation*}
and then
\begin{equation*}
\nabla_{i}u\nabla_{j}\Delta u
+V^{k}\nabla_{k}\nabla_{j}u
\nabla_{i}u=\nabla_{i}u\nabla_{j}\Delta_{V}u
-\nabla_{i}u\nabla_{k}u\nabla_{j}V^{k}.
\end{equation*}
On the other hand, we have
\begin{equation*}
R_{j}{}^{k}\frac{\nabla_{k}u\nabla_{i}u}{u^{2}(1-f)^{2}}
=({\rm Ric}_{V})_{j}{}^{k}
\frac{\nabla_{k}u\nabla_{i}u}{u^{2}(1-f)^{2}}
+\frac{\nabla_{i}u\nabla^{k}u
(\nabla_{j}V_{k}+\nabla_{k}V_{j})}{2u^{2}(1-f)^{2}}.
\end{equation*}
Similarly, we can find an analogue identity for $R_{i}{}^{k}\nabla_{k}u\nabla_{j}u/u^{2}(1-f)^{2}$. Therefore
\begin{eqnarray*}
\Delta_{V}w_{ij}&=&\frac{\nabla_{i}u[\nabla_{j}\Delta_{V}u
+({\rm Ric}_{V})_{j}{}^{k}\nabla_{k}u]}{u^{2}(1-f)^{2}}
+\frac{\nabla_{j}u[\nabla_{i}\Delta_{V}u+({\rm Ric}_{V})_{i}{}^{k}
\nabla_{k}u]}{u^{2}(1-f)^{2}}\\
&&+ \ \frac{2\nabla_{i}\nabla_{k}u\nabla_{j}\nabla^{k}u}{u^{2}(1-f)^{2}}
+\frac{4f\nabla^{k}u(\nabla_{i}u\nabla_{j}
\nabla_{k}u+\nabla_{j}u\nabla_{i}\nabla_{k}u)}{u^{3}(1-f)^{3}}\\
&&+ \ \frac{2\nabla_{i}u\nabla_{j}u(\langle \nabla u,\nabla f\rangle
+f\Delta_{V}u)}{u^{3}(1-f)^{3}}+\frac{6f^{2}|\nabla u|^{2}
\nabla_{i}u\nabla_{j}u}{u^{4}(1-f)^{4}}\\
&&+ \ \frac{\nabla_{i}u\nabla^{k}u}{u^{2}(1-f)^{2}}
\frac{\nabla_{k}V_{j}-\nabla_{j}V_{k}}{2}
+\frac{\nabla_{j}u\nabla^{k}u}{u^{2}(1-f)^{2}}
\frac{\nabla_{k}V_{i}-\nabla_{i}V_{k}}{2}.
\end{eqnarray*}
Together with the expression of $\partial_{t}w_{ij}$, we arrive at
\begin{eqnarray*}
\left(\partial_{t}-\Delta_{V}\right)w_{ij}
&=&-({\rm Ric}_{V})_{i}{}^{k}\frac{\nabla_{k}u
\nabla_{j}u}{u^{2}(1-f)^{2}}
-({\rm Ric}_{V})_{j}{}^{k}
\frac{\nabla_{k}u\nabla_{i}u}{u^{2}(1-f)^{2}}\\
&&- \ \frac{2\nabla_{i}\nabla_{k}u
\nabla_{j}\nabla^{k}u}{u^{2}(1-f)^{2}}
-\frac{4f\nabla^{k}u(\nabla_{i}u
\nabla_{j}\nabla_{k}u+\nabla_{j}u\nabla_{i}\nabla_{k}u)}{u^{3}(1-f)^{3}}\\
&&-\frac{2\nabla_{i}u\nabla_{j}u
\langle\nabla u,\nabla f\rangle}{u^{3}(1-f)^{3}}
-\frac{6f^{2}|\nabla u|^{2}\nabla_{i}u
\nabla_{j}u}{u^{4}(1-f)^{4}}\\
&&- \ \frac{\nabla_{i}u\nabla^{k}u}{u^{2}(1-f)^{2}}
\frac{\nabla_{k}V_{j}-\nabla_{j}V_{k}}{2}
-\frac{\nabla_{j}u\nabla^{k}u}{u^{2}(1-f)^{2}}
\frac{\nabla_{k}V_{i}-\nabla_{i}V_{k}}{2}
\end{eqnarray*}
As in \cite{HanZhang12}, the middle four terms $H$ on the right-hand side can be written as
\begin{eqnarray*}
H&=&-\frac{2\nabla_{i}\nabla_{k}u
\nabla_{j}\nabla^{k}u}{u^{2}(1-f)^{2}}
-\frac{4f\nabla^{k}u(\nabla_{i}u
\nabla_{j}\nabla_{k}u+\nabla_{j}u\nabla_{i}\nabla_{k}u)}{u^{3}(1-f)^{3}}\\
&&-\frac{2\nabla_{i}u\nabla_{j}u
\langle\nabla u,\nabla f\rangle}{u^{3}(1-f)^{3}}
-\frac{6f^{2}|\nabla u|^{2}\nabla_{i}u
\nabla_{j}u}{u^{4}(1-f)^{4}}\\
&=&-\frac{2f}{1-f}\nabla^{k}f\nabla_{k}w_{ij}
-\frac{2|\nabla f|^{2}}{1-f}w_{ij}
-2(v_{ik}+fw_{ik})(v_{j}{}^{k}+fw_{j}{}^{k}).
\end{eqnarray*}
Plugging the expression of $H$ into $(\partial_{t}-\Delta_{V})w_{ij}$ we
obtain the result.
\end{proof}

From (\ref{6.7}) we see that
\begin{equation*}
w=\frac{|\nabla f|^{2}}{(1-f)^{2}}
\end{equation*}
so that Lemma \ref{l6.2} and Lemma \ref{l6.3} can be rewritten as
\begin{eqnarray*}
\left(\partial_{t}-\Delta_{V}\right)
v_{ij}&=&-\frac{2f}{1-f}\nabla^{k}f\nabla_{k}v_{ij}
-(1-f)w v_{ij}+2R_{kij\ell}v^{k\ell}-({\rm Riv}_{V})_{ik}v_{j}{}^{k}\\
&&- \ ({\rm Ric}_{V})_{jk}v_{i}{}^{k}+v_{i}{}^{k}(\mathscr{A}_{V}g)_{jk}+v_{j}{}^{k}
(\mathscr{A}_{V}g)_{ik}\\
&&- \ \frac{\nabla^{k}u}{u(1-f)}\bigg(\nabla_{i}({\rm Ric}_{V})_{jk}+\nabla_{j}({\rm Ric}_{V})_{ik}
-\nabla_{k}({\rm Ric}_{V})_{ij}\bigg),\\
\left(\partial_{t}-\Delta_{V}\right)w_{ij}&=&-\frac{2f}{1-f}\nabla^{k}f\nabla_{k}w_{ij}
-2(1-f)ww_{ij}-2(v_{ik}+fw_{ik})(v_{j}{}^{k}
+fw_{j}{}^{k})\\
&&- \ ({\rm Ric}_{V})_{ik}w_{j}{}^{k}
-({\rm Ric}_{V})_{jk}w_{i}{}^{k}+w_{i}{}^{k}(\mathscr{A}_{V}g)_{jk}
+w_{j}{}^{k}(\mathscr{A}_{V}g)_{ik},
\end{eqnarray*}
where $\mathscr{A}_{V}g$ stands for the tensor field given by
\begin{equation}
(\mathscr{A}_{V}g)_{ij}:=\frac{\nabla_{i}V_{j}-\nabla_{j}V_{i}}{2}.\label{6.13}
\end{equation}
The tensor field exactly the $2$-form $dV_{\flat}$ where $V_{\flat}$
is the corresponding $1$-form of $V$. When $V$ is a gradient vector field $
V=\nabla\phi$, we see that $\mathscr{A}_{V}g$ vanishes identically on
$\mathcal{M}$. In this sense $\mathscr{A}_{V}g$ is an obstruction of $V$
being gradient.

Let $p\in\mathcal{M}$ and choose a local orthonormal coordinates
$(x^{i})_{1\leq i\leq m}$ around $p$. We follow the method in
\cite{HanZhang12}. Consider the operator
\begin{equation}
\square_{V}:=\partial_{t}-\Delta_{V}+\frac{2f}{1-f}\langle\nabla f,\nabla\!\
\rangle.\label{6.14}
\end{equation}
Then the matrices $\boldsymbol{V}=(v_{ij})$ and $\boldsymbol{W}=(w_{ij})$ satisfy
\begin{eqnarray}
\square_{V}\boldsymbol{V}&=&-(1-f)w\boldsymbol{V}-\boldsymbol{P}-\boldsymbol{VA}
+\boldsymbol{AV},\label{6.15}\\
\square_{V}\boldsymbol{W}&=&-2(1-f)w\boldsymbol{W}
-2(\boldsymbol{V}+f\boldsymbol{W})^{2}-\boldsymbol{Q}-\boldsymbol{WA}
+\boldsymbol{AW},\label{6.16}
\end{eqnarray}
where $\boldsymbol{P}=(P_{ij}), \boldsymbol{Q}=(Q_{ij}), \boldsymbol{A}
=(A_{ij})$ are matrices whose entries are
\begin{eqnarray}
P_{ij}&:=&-2\sum_{1\leq k,\ell\leq m}R_{kij\ell}v_{k\ell}+
\sum_{1\leq k\leq m}\bigg[({\rm Ric}_{V})_{ik}v_{kj}+
v_{ik}({\rm Ric}_{V})_{kj}\nonumber\\
&&+ \ \frac{\nabla_{k}u}{u(1-f)}\bigg(\nabla_{i}({\rm Ric}_{V})_{jk}
+\nabla_{j}({\rm Ric}_{V})_{ik}-\nabla_{k}({\rm Ric}_{V})_{ij}\bigg)\bigg],\label{6.17}\\
Q_{ij}&:=&\sum_{1\leq k\leq m}\bigg(({\rm Ric}_{V})_{ik}w_{kj}
+w_{ik}({\rm Ric}_{V})_{kj}\bigg),\label{6.18}\\
A_{ij}&:=&(\mathscr{A}_{V}g)_{ij}.\label{6.20}
\end{eqnarray}
For any real number $\alpha$ we define
\begin{equation}
\boldsymbol{V}\oplus_{\alpha}\boldsymbol{W}:=
\alpha\boldsymbol{V}+\boldsymbol{W}.\label{6.20}
\end{equation}
Then
\begin{eqnarray}
\square_{V}(\boldsymbol{V}\oplus_{\alpha}
\boldsymbol{W})&=&-\alpha(1-f)w\boldsymbol{V}
-2(1-f)w\boldsymbol{V}-2(\boldsymbol{V}+f\boldsymbol{W})^{2}\nonumber\\
&&- \ \boldsymbol{P}\oplus_{\alpha}\boldsymbol{Q}
-(\boldsymbol{V}\oplus_{\alpha}\boldsymbol{W})\boldsymbol{A}
+\boldsymbol{A}(\boldsymbol{V}\oplus_{\alpha}\boldsymbol{W}).\label{6.21}
\end{eqnarray}
Let $\boldsymbol{\xi}\in T_{p}\mathcal{M}\cong{\bf R}^{m}$ be a unit eigenvector of $\boldsymbol{V}
\oplus_{\alpha}\boldsymbol{W}$, i.e., $(\boldsymbol{V}\oplus_{\alpha}\boldsymbol{W}
)\boldsymbol{\xi}=\lambda\boldsymbol{\xi}$. By parallel translation along
geodesics, we extend $\boldsymbol{\xi}$ to a smooth vector field, still
denoted by $\boldsymbol{\xi}$, near $p$. Then
\begin{equation}
\lambda=(\boldsymbol{V}\oplus_{\alpha}\boldsymbol{W})(\boldsymbol{\xi},
\boldsymbol{\xi})\label{6.22}
\end{equation}
is a smooth function near $p$. From (\ref{6.21}) and (\ref{6.22}) we obtain
\begin{eqnarray*}
\square_{V}\lambda&=&-\alpha(1-f)w\boldsymbol{V}(\boldsymbol{\xi},
\boldsymbol{\xi})-2(1-f)w\boldsymbol{W}(\boldsymbol{\xi},\boldsymbol{\xi})
-2|(\boldsymbol{V}+f\boldsymbol{W})\boldsymbol{\xi}|^{2}\\
&&- \ (\boldsymbol{P}\oplus_{\alpha}\boldsymbol{Q})(\boldsymbol{\xi},
\boldsymbol{\xi})-((\boldsymbol{V}\oplus_{\alpha}\boldsymbol{W})
\boldsymbol{A})(\boldsymbol{\xi},\boldsymbol{\xi})
+(\boldsymbol{A}(\boldsymbol{V}\oplus_{\alpha}
\boldsymbol{W}))(\boldsymbol{\xi},\boldsymbol{\xi})\\
&\leq&-\frac{2\lambda^{2}}{\alpha^{2}}
-\lambda\left(w-\frac{4}{\alpha^{2}}
\boldsymbol{W}(\boldsymbol{\xi},\boldsymbol{\xi})\right)
+f\lambda\left(w-\frac{4}{\alpha}\boldsymbol{W}
(\boldsymbol{\xi},\boldsymbol{\xi})\right)-(\boldsymbol{P}\oplus_{\alpha}\boldsymbol{Q})(\boldsymbol{\xi},
\boldsymbol{\xi})
\end{eqnarray*}
where we used the estimate (2.6) in \cite{HanZhang12} and
\begin{equation*}
((\boldsymbol{V}\oplus_{\alpha}\boldsymbol{W})\boldsymbol{A})(
\boldsymbol{\xi},\boldsymbol{\xi})
=\lambda\boldsymbol{A}(\boldsymbol{\xi},\boldsymbol{\xi})
=(\boldsymbol{A}(\boldsymbol{V}\oplus_{\alpha}
\boldsymbol{W}))(\boldsymbol{\xi},\boldsymbol{\xi}).
\end{equation*}
Since $\boldsymbol{W}(\xi,\xi)
\leq w$, it follows from (2.7) in \cite{HanZhang12} that
\begin{equation}
\square_{V}\lambda\leq-\frac{2\lambda^{2}}{\alpha^{2}}
-(\boldsymbol{P}\oplus_{\alpha}\boldsymbol{Q})(\boldsymbol{\xi},
\boldsymbol{\xi}) \ \ \ \text{at} \ p, \ \ \ \text{whenever} \ \lambda\geq0,\label{6.23}
\end{equation}
where $\alpha\geq4$.

{\bf Proof part (a) of Theorem \ref{t6.1}:} As in \cite{HanZhang12}, we consider the quantity
\begin{equation}
\boldsymbol{V}\oplus_{\alpha,\tau}\boldsymbol{W}
:=\alpha\boldsymbol{V}+\boldsymbol{W}-\frac{\tau}{t}\boldsymbol{g}\label{6.24}
\end{equation}
where $\boldsymbol{g}:=(g_{ij})$ and $\tau$ is a positive constant determined
later. Assume now that $\boldsymbol{V}\oplus_{\alpha,\tau}\boldsymbol{W}$ has the
largest nonnegative eigenvalue with the unit eigenvector $\boldsymbol{\xi}$ at a point $(p_{1},t_{1})$ with $t_{1}>0$. As before we consider
\begin{equation*}
\lambda:=(\boldsymbol{V}\oplus_{\alpha}\boldsymbol{W})(\boldsymbol{\xi},
\boldsymbol{\xi}), \ \ \ \mu:=(\boldsymbol{V}\oplus_{\alpha,\tau}
\boldsymbol{W})(\boldsymbol{\xi},\boldsymbol{\xi})=\lambda-\frac{\tau}{t}.
\end{equation*}
Since $\mu$ has its nonnegative maximum at $(p_{1},t_{1})$, it follows that
$\Delta\mu\leq0=\nabla\mu\leq\partial_{t}\mu$ and hence $\square_{V}\mu\leq0$ at $(p_{1},t_{1})$. Consequently,
\begin{equation}
\frac{2\lambda^{2}}{\alpha^{2}}
\leq\frac{\tau}{t^{2}}+|(\boldsymbol{P}\oplus_{\alpha}\boldsymbol{Q})(\boldsymbol{\xi},\boldsymbol{\xi})
| \ \ \ \text{at} \ (p_{1},t_{1})\label{6.25}
\end{equation}
as that of (2.11) in \cite{HanZhang12}. Let $\boldsymbol{\xi}=(\xi_{1},\cdots,\xi^{m})^{T}$ and note that
\begin{eqnarray*}
&&|(\boldsymbol{P}\oplus_{\alpha}\boldsymbol{Q})(\boldsymbol{\xi},
\boldsymbol{\xi})| \ \ \leq \ \ \alpha|\boldsymbol{P}(\boldsymbol{\xi},
\boldsymbol{\xi})|+|\boldsymbol{Q}(\boldsymbol{\xi},\boldsymbol{\xi})|\\
&\leq&\alpha\left|\sum_{1\leq i,j\leq m}
\xi_{i}\xi_{j}\left(-2\sum_{1\leq k,\ell\leq m}R_{kij\ell}v_{k\ell}
+\sum_{1\leq k\leq m}R^{V}_{ik}v_{kj}
+\sum_{1\leq k\leq m}v_{ik}R^{V}_{kj}\right)\right|\\
&&+ \ \left|\sum_{1\leq i,j,k\leq m}\xi_{i}\xi_{j}
\left(R^{V}_{ik}w_{kj}+w_{ik}R^{V}_{kj}\right)\right|\\
&&+ \ \alpha\left|\sum_{1\leq i,j,k\leq m}
\xi_{i}\xi_{j}\frac{\nabla_{k}u}{u(1-f)}
\left(\nabla_{i}R^{V}_{jk}+\nabla_{j}R^{V}_{ik}-\nabla_{k}R^{V}_{ij}
\right)\right|
\end{eqnarray*}
where $R^{V}_{ij}:=({\rm Ric}_{V})_{ij}$. Since $\boldsymbol{\xi}$ is unit,
it follows that
\begin{eqnarray*}
\left|\sum_{ 1\leq i,j,k\leq m}\xi_{i}\xi_{j}\left(R^{V}_{ik}w_{kj}
+w_{ik}R^{V}_{kj}\right)\right|&\leq&\sum_{1\leq i,j,k\leq m}
\left|R^{V}_{ik}w_{kj}+w_{ik}R^{V}_{kj}\right|\\
&\leq&2\left(\sum_{1\leq i,j,k\leq m}(R^{V}_{ik})^{2}\right)^{\frac{1}{2}}
\left(\sum_{1\leq i,j,k\leq m}w^{2}_{kj}\right)^{\frac{1}{2}}\\
&\leq&2m|{\rm Ric}_{V}||\boldsymbol{W}|.
\end{eqnarray*}
Similarly,
\begin{equation*}
\left|\sum_{1\leq i,j,k\leq m}
\xi_{i}\xi_{j}\frac{\nabla_{k}u}{u(1-f)}
\left(\nabla_{i}R^{V}_{jk}+\nabla_{j}R^{V}_{ik}-\nabla_{k}R^{V}_{ij}
\right)\right|\leq3m|\nabla{\rm Ric}_{V}||\boldsymbol{W}|^{1/2}.
\end{equation*}
As the inequality (2.12) in \cite{HanZhang12}, we have
\begin{eqnarray}
|(\boldsymbol{P}\oplus_{\alpha}\boldsymbol{Q})(\boldsymbol{\xi},
\boldsymbol{\xi})|&\leq&\left|\sum_{1\leq i,j\leq m}
\xi_{i}\xi_{j}\left(-2\sum_{1\leq k,\leq m}R_{kij\ell}(\alpha v_{k\ell}
+w_{k\ell})\right.\right.\nonumber\\
&&+ \ \left.\left.\sum_{1\leq k\leq m}R^{V}_{ik}(\alpha v_{kj}
+w_{kj})+\sum_{1\leq k\leq m}(\alpha v_{ik}
+w_{ik})R^{V}_{kj}\right)\right|\nonumber\\
&&+ \ \left|\sum_{1\leq i,j\leq m}\xi_{i}\xi_{j}
\left(-2\sum_{1\leq k,\ell\leq m}R_{kij\ell}w_{k\ell}
+\sum_{1\leq k\leq m}R^{V}_{ik}w_{kj}\right.\right.\label{6.26}\\
&&+ \ \left.\left.\sum_{1\leq k\leq m}w_{ik}R^{V}_{kj}\right)\right|
+3m|\nabla{\rm Ric}_{V}||\boldsymbol{W}|^{1/2}+2m|{\rm Ric}_{V}||\boldsymbol{W}|.\nonumber
\end{eqnarray}
In order to bound the function $|(\boldsymbol{P}\oplus_{\alpha}
\boldsymbol{Q})(\boldsymbol{\xi},\boldsymbol{\xi})|$ at the point $p_{1}$,
as in \cite{HanZhang12}, we choose a local coordinate system so that the
matrix $\boldsymbol{V}\oplus_{\alpha}\boldsymbol{W}$ is diagonal and $\boldsymbol{V}\oplus_{\alpha}\boldsymbol{W}-\frac{\tau}{t}\boldsymbol{g}
={\rm diag}(\mu_{1},\cdots,\mu_{m})$ with $\mu_{1}\leq\cdots\leq \mu_{m}$
and $\mu_{1}<0<\mu_{m}$. Then
\begin{eqnarray*}
&&\left|\sum_{1\leq i,j,k,\ell\leq m}\xi_{i}\xi_{j}R_{kij\ell}(\alpha v_{k\ell}
+w_{k\ell})\right|\\
&\leq&\sum_{1\leq i,j,k,\ell\leq m}\left|R_{kij\ell}
\left(\alpha v_{k\ell}+w_{k\ell}-\frac{\tau}{t}g_{k\ell}\right)\right|
+\sum_{1\leq i,j,k,\ell\leq m}\left|R_{kij\ell}g_{k\ell}\right|\frac{\tau}{t}\\
&=&\sum_{1\leq i,j,k\leq m}\left|R_{kijk}\mu_{k}\right|
+\sum_{1\leq i,j,k\leq m}\left|R_{kijk}\right|\frac{\tau}{t}\\
&\leq&\left(\sum_{1\leq i,j,k\leq m}R^{2}_{kijk}\right)^{1/2}
\left[\left(\sum_{1\leq i,j,k\leq m}\mu^{2}_{k}\right)^{1/2}
+\left(\sum_{1\leq i,j,k\leq m}1\right)^{1/2}\frac{\tau}{t}\right]\\
&\leq&|{\rm Rm}|\left(m\left(\sum_{1\leq k\leq m}\mu^{2}_{k}\right)^{1/2}
+m^{3/2}\frac{\tau}{t}\right)\\
&\leq&|{\rm Rm}|\left(m^{3/2}(\mu_{m}+|\mu_{1}|)+m^{3/2}\frac{\tau}{t}\right) \ \ = \ \ m^{3/2}|{\rm Rm}|\left(\mu_{m}
+|\mu_{1}|+\frac{\tau}{t}\right).
\end{eqnarray*}
Here we used the estimate that
\begin{equation*}
\left(\sum_{1\leq k\leq m}\mu^{2}_{k}\right)^{\frac{1}{2}}
\leq\bigg((m-i)\mu^{2}_{m}+i\mu^{2}_{1}\bigg)^{\frac{1}{2}}\leq\bigg(\sqrt{m-i}\mu_{m}
+\sqrt{i}|\mu_{1}|\bigg)^{\frac{1}{2}}\leq\sqrt{m}(\mu_{m}+|\mu_{1}|)
\end{equation*}
where $\mu_{i}$ is the largest eigenvalue so that $\mu_{i}<0$ but
$\mu_{i+1}\geq0$. Similarly, we have
\begin{eqnarray*}
&&\left|\sum_{1\leq i,j,k\leq m}\xi_{i}\xi_{j}R^{V}_{ik}(\alpha v_{kj}+w_{kj})\right|\\
&\leq&\sum_{1\leq i,j,k\leq m}\left|R^{V}_{ik}\left(\alpha v_{kj}+w_{kj}-\frac{\tau}{t}g_{kj}\right)
\right|+\sum_{1\leq i,j,k\leq m}|R^{V}_{ik}g_{kj}|\frac{\tau}{t}\\
&=&\sum_{1\leq i,j\leq m}|R^{V}_{ij}\mu_{j}|+\sum_{1\leq i,j\leq m}
|R^{V}_{ij}|\frac{\tau}{t}\\
&\leq&\left(\sum_{1\leq i,j\leq m}|R^{V}_{ij}|^{2}\right)^{1/2}
\left[\left(\sum_{1\leq i,j\leq m}\mu^{2}_{j}\right)^{1/2}
+m\frac{\tau}{t}\right]\\
&\leq&m|{\rm Ric}_{V}|\left[\sqrt{m}(\mu_{m}+|\mu_{1}|)+\frac{\tau}{t}\right] \ \ \leq \ \ m^{3/2}|{\rm Ric}_{V}|\left(\mu_{m}+|\mu_{1}|+\frac{\tau}{t}\right).
\end{eqnarray*}
Plugging those estimates into (\ref{6.26}) yields
\begin{eqnarray}
|(\boldsymbol{P}\oplus_{\alpha}\boldsymbol{Q})(\boldsymbol{\xi},
\boldsymbol{\xi})|&\leq&2m^{3/2}\bigg(|{\rm Rm}|+|{\rm Ric}_{V}|\bigg)
\bigg(\mu_{m}+|\mu_{1}|+\frac{\tau}{t}\bigg)\nonumber\\
&&+ \ 3m|\nabla{\rm Ric}_{V}||\boldsymbol{W}|^{1/2}
+4m\bigg(|{\rm Rm}|+|{\rm Ric}_{V}|\bigg)
|\boldsymbol{W}|.\label{6.27}
\end{eqnarray}
Set
\begin{equation}
K_{1}:=\max_{\mathcal{M}}\bigg(|{\rm Rm}|+|{\rm Ric}_{V}|\bigg), \ \ \
K_{2}:=\max_{\mathcal{M}}|\nabla{\rm Ric}_{V}|.\label{6.28}
\end{equation}
Therefore, using $2|\boldsymbol{W}|^{1/2}\leq 1+|\boldsymbol{W}|$, we arrive at
\begin{equation}
|(\boldsymbol{P}\oplus_{\alpha}\boldsymbol{Q})(\boldsymbol{\xi},\boldsymbol{\xi})
|\leq 2m^{3/2}K_{1}\bigg(\mu_{m}+|\mu_{1}|+\frac{\tau}{t}\bigg)
+2m K_{2}+4m(K_{1}+K_{2})|\boldsymbol{W}|.\label{6.29}
\end{equation}
By the page 9 in \cite{HanZhang12}, we have
\begin{equation*}
\mu_{m}+|\mu_{1}|\leq m\mu_{m}-\frac{\alpha\Delta u}{u(1-f)}
-\frac{|\nabla u|^{2}}{u^{2}(1-f)^{2}}+\frac{m\tau}{t}.
\end{equation*}
By (5.5), we deduce that
\begin{equation*}
-\frac{\alpha\Delta u}{u}\leq\frac{n\alpha^{2}}{2t}+\frac{n\alpha^{2}K}{\alpha-1}
+\frac{\alpha}{u}\langle V,\nabla u\rangle-\frac{|\nabla u|^{2}}{u^{2}}
\leq\frac{n\alpha^{2}}{2t}+\frac{n\alpha^{2}K}{\alpha-1}
+\frac{\alpha^{2}}{2}|V|^{2}-\frac{|\nabla u|^{2}}{2u^{2}}.
\end{equation*}
Since $1/(1-f)\leq1$ it follows that
\begin{eqnarray}
|(\boldsymbol{P}\oplus_{\alpha}\boldsymbol{Q})(\boldsymbol{\xi},\boldsymbol{\xi})|
&\leq&2m^{3/2}K_{1}\left(m\mu_{m}+\frac{n\alpha^{2}+2\tau}{2t}\right)+4m(K_{1}
+K_{2})|\boldsymbol{W}|\nonumber\\
&&+ \ 2m^{3/2}K_{1}\left(\frac{n\alpha^{2}K}{\alpha-1}+\frac{\alpha^{2}}{2}
|V|^{2}\right)+2mK_{2}\label{6.30}
\end{eqnarray}
at the point $(p_{1},t_{1})$. Because ${\rm Riv}^{n,m}_{V}\geq-K$ implies
${\rm Ric}_{V}\geq-K$, the estimate (5.7) tells us that
\begin{equation*}
|\boldsymbol{W}|=\frac{\nabla u|^{2}}{u^{2}(1-f)^{2}}
\leq\left(\frac{1}{t}+2K\right)\frac{-f}{(1-f)^{2}}\leq\frac{1}{4}
\left(\frac{1}{t}+2K\right).
\end{equation*}
Since $\mu=\mu_{m}<\lambda$ at $(p_{1},t_{1})$, by the same argument in the
page 10 of \cite{HanZhang12}, we obtain
\begin{eqnarray}
\frac{2\lambda^{2}}{\alpha^{2}}
&\leq&\frac{\tau}{t^{2}}+m^{3/2}(K_{1}+K_{2})\left(2m\lambda+\frac{n\alpha^{2}
+2\tau+1}{t}\right)
+m^{3/2}\alpha^{2}|V|^{2}K_{1}\nonumber\\
&&+ \ 2mK_{2}+mKK_{2}+
\left(\frac{2m^{3/2}n\alpha^{2}}{\alpha-1}+m\right)KK_{1}\label{6.31}
\end{eqnarray}
from (\ref{6.25}), at the point $(p_{1},t_{1})$. By assumption $n\geq m$ and $\alpha\geq4$, we have
\begin{equation*}
n\alpha^{2}+2\tau+1\leq n\alpha^{2}+\alpha^{2}\tau+\alpha^{2}
\leq(n+1)\alpha^{2}(1+\tau)
\end{equation*}
and hence
\begin{equation*}
m^{3/2}(K_{1}+K_{2})\left(2m\lambda+\frac{n\alpha^{2}+2\tau+1}{t}\right)
\leq 2n m^{3/2}\alpha^{2}(K_{1}+K_{2})\left(\lambda+\frac{1+\tau}{t}\right).
\end{equation*}
Letting
\begin{eqnarray*}
B_{1}&:=&2nm^{3/2}\alpha^{2}(K_{1}+K_{2}), \\
B_{2}&:=&m^{3/2}\alpha^{2}|V|^{2}K_{1}+2mK_{2}+mKK_{2}+
\left(\frac{2m^{3/2}n\alpha^{2}}{\alpha-1}
+m\right)KK_{1}
\end{eqnarray*}
we conclude from (\ref{6.31}) that
\begin{equation}
\frac{2\lambda^{2}}{\alpha^{2}}
\leq\frac{\tau}{t^{2}}+B_{1}\left(\alpha
\frac{\lambda}{\alpha}+\frac{1+\tau}{t}\right)+B_{2}.\label{6.32}
\end{equation}
By Cauchy's inequality, we get $B_{1}\lambda\leq\frac{\lambda^{2}}{\alpha^{2}}
+\frac{\alpha^{2}B^{2}_{1}}{4}$ and hence
\begin{equation*}
\frac{\lambda^{2}}{\alpha^{2}}\leq\frac{\tau+1}{t^{2}}
+\frac{B_{1}\sqrt{\tau+1}}{2}2\frac{\sqrt{\tau+1}}{t}+B_{2}+\frac{\alpha^{2}B^{2}_{1}}{4}.
\end{equation*}
Putting
\begin{equation*}
B:=\max\left\{\frac{B_{1}\sqrt{\tau+1}}{2},\sqrt{B_{2}+\frac{1}{4}\alpha^{2}
B^{2}_{1}}\right\}
\end{equation*}
the above inequality yields
\begin{equation}
\frac{\lambda}{\alpha}\leq\frac{\sqrt{\tau+1}}{t}+B\label{6.33}
\end{equation}
at the point $(p_{1},t_{1})$. As in the page 10 of \cite{HanZhang12}, we then arrive at
\begin{equation*}
(\boldsymbol{V}\oplus_{\alpha}\boldsymbol{W})(\boldsymbol{\eta},
\boldsymbol{\eta})-\frac{\tau}{t}\leq\left(\lambda-\frac{\tau}{t}\right)_{(p_{1},t_{1})}
\leq\frac{\alpha\sqrt{\tau+1}-\tau}{t}+\alpha B
\end{equation*}
in $\mathcal{M}\times(0,T]$. If we choose $\alpha:=\frac{\tau}{\sqrt{\tau+1}}\geq4$, then
\begin{equation*}
t|\nabla^{2}u|\leq\left(\sqrt{\tau+1}+Bt\right)u\left(1-\ln\frac{u}{A}\right)
\end{equation*}
where $0<u\leq A$ and $\tau\geq 4\sqrt{\tau+1}$. The restriction on $\tau$
implies that $\tau\geq 8+4\sqrt{5}$ and that we can take $\tau:=8+4\sqrt{5}$
and then $\alpha=4$. Hence
\begin{equation*}
t|\nabla^{2}u|\leq\left(2+\sqrt{5}+Bt\right) u\left(1-\ln\frac{u}{A}\right)
\end{equation*}
where we can take $B$ to be the constant
\begin{equation*}
B:=\sqrt{16 m^{3/2}|V|^{2}K_{1}
+2m K_{2}+3m KK_{2}+14 m^{3/2}n KK_{1}
+100 n^{2}m^{3}(K_{1}+K_{2})^{2}}.
\end{equation*}

{\bf Proof part (b) of Theorem \ref{t6.1}:} Consider the cutoff function $\psi$ constructed in \cite{HanZhang12}, which
is supported in $Q_{R,T}(x_{0},t_{0})$, equals $1$ in $Q_{R/2,T/2}(x_{0},
t_{0})$, and satisfies
\begin{equation*}
|\nabla\psi|\leq\frac{C}{R}, \ |\Delta_{V}\psi|\leq C\frac{1+R\sqrt{K}}{R^{2}}, \
\frac{|\partial_{t}\psi|}{\sqrt{\psi}}\leq\frac{C}{T}, \ \frac{|\nabla\psi|^{2}}{\psi}
\leq\frac{C}{R^{2}}
\end{equation*}
where $C$ is a positive constant depending only on $n$. As in
\cite{HanZhang12}, we may require that $t_{0}=T$ and $\psi$ is supported in the slightly shorter space time cube $Q_{R,3T/4}(x_{0},t_{0})$.

For any smooth function $\eta$, as in the page 11 of \cite{HanZhang12}, we have
\begin{equation}
\Box_{V,\psi}(\psi\eta)=\psi\Box_{V}\eta+\eta\square_{V,\psi}\psi\label{6.34}
\end{equation}
where
\begin{equation}
\square_{V,\psi}:=
\square_{V}+\frac{2}{\psi}\langle\nabla\psi,\nabla\!\ \rangle.\label{6.35}
\end{equation}
Choosing $\eta=\lambda$ defined in \ref{6.22} and using the evolution equation
of $\lambda$, we have
\begin{eqnarray}
\square_{V,\psi}(\psi\lambda)&=&-\psi[H+(\boldsymbol{P}\oplus_{\alpha}
\boldsymbol{Q})(\boldsymbol{\xi},\boldsymbol{\xi})]+\lambda\square_{V,\psi}\psi
\nonumber\\
&&- \ \psi((\boldsymbol{V}\oplus_{\alpha}\boldsymbol{W})\boldsymbol{A})(
\boldsymbol{\xi},\boldsymbol{\xi})+\psi(\boldsymbol{A}(\boldsymbol{V}\oplus_{\alpha}
\boldsymbol{W}))(\boldsymbol{\xi},\boldsymbol{\xi}),\label{6.36}
\end{eqnarray}
where
\begin{equation}
H:=\alpha(1-f)w\boldsymbol{V}(\boldsymbol{\xi},\boldsymbol{\xi})
+2(1-f)w\boldsymbol{W}(\boldsymbol{\xi},\boldsymbol{\xi})
+2|(\boldsymbol{V}+f\boldsymbol{W})\boldsymbol{\xi}|^{2}.\label{6.37}
\end{equation}
Given a positive constant $\beta$, consider a unit eigenvector $\boldsymbol{\xi}$
of $\psi(\boldsymbol{V}\oplus_{\alpha}\boldsymbol{W})+\beta f\boldsymbol{g}$
with the maximal eigenvalue $\mu_{m}$ at the point $(p_{1},x_{1})$. Extend $\boldsymbol{\xi}$ to be a vector field, still
denoted by $\boldsymbol{\xi}$, by parallel translation along geodesics from $p_{1}$. Let $\mu_{1},\cdots,\mu_{m}$ be the eigenvalues of the two form $\psi(\boldsymbol{V}
\oplus_{\alpha}\boldsymbol{W})+\beta f\boldsymbol{g}$ at $(p_{1},t_{1})$ with the
increasing order. As before, we may assume that $\mu_{1}<0<\mu_{m}$. Define
\begin{equation}
\mu:=[\psi\boldsymbol{V}\oplus_{\alpha}\boldsymbol{W}
+\beta f\boldsymbol{g}](\boldsymbol{\xi},\boldsymbol{\xi})
=\psi\lambda+\beta f.
\end{equation}
Note that $\mu=\mu_{m}$ at the point $(p_{1},t_{1})$. From (\ref{6.36}) we get
\begin{equation}
\psi\square_{V,\psi}\mu=-\psi^{2}[H^{2}+(\boldsymbol{P}\oplus_{\alpha}
\boldsymbol{Q})(\boldsymbol{\xi},\boldsymbol{\xi})]+\psi\lambda\square_{V,\psi}
\psi+\psi\beta\square_{V,\psi}f.\label{6.39}
\end{equation}
By definition, $\square_{V,\psi}\psi$ is equal to
\begin{eqnarray*}
\square_{V,\psi}\psi&=&\partial_{t}\psi-\Delta_{V}\psi
+\frac{2f}{1-f}\langle\nabla f,\nabla\psi\rangle
+\frac{2}{\psi}|\nabla\psi|^{2}\\
&=&\partial_{t}\psi-\Delta_{V}\psi
+\frac{2f}{1-f}\left\langle\sqrt{\psi}\nabla f,\frac{\nabla\psi}{\sqrt{\psi}}
\right\rangle+\frac{2}{\psi}|\nabla\psi|^{2}
\end{eqnarray*}
Without loss of generality, we may assume that $0<u\leq A/e^{3}$; otherwise,
for $\frac{A}{e^{3}}\leq u\leq A$ we can consider a new function $\tilde{u}:=
u/e^{3}\in(0,A/e^{3}]$ and hence $\tilde{u}$ also satisfies the same
estimate (\ref{6.2}) which implies (\ref{6.2}) for $u$. Under our hypothesis and (\ref{6.5}), we arrive at
\begin{equation*}
\square_{V}f=\partial_{t}f-\Delta_{V}f+\frac{2f}{1-f}|\nabla f|^{2}
=|\nabla f|^{2}+\frac{2f}{1-f}|\nabla f|^{2}=\frac{1+f}{1-f}|\nabla f|^{2}\leq-\frac{1}{2}
|\nabla f|^{2}.
\end{equation*}
Consequently,
\begin{equation*}
\psi\square_{V,\psi}f=\psi\square_{V}f+2\left\langle\frac{\nabla\psi}{\sqrt{\psi}},
\sqrt{\psi}\nabla f\right\rangle
\leq-\frac{1}{4}\psi|\nabla f|^{2}+4\frac{|\nabla\psi|^{2}}{\psi}
\end{equation*}
As the estimate (\ref{6.23}) (or see the page 13 in \cite{HanZhang12}) we have (since $\mu\geq0$ implies $\psi\lambda\geq-\beta f\geq0$)
\begin{equation*}
-\psi^{2}H\leq-\frac{2(\psi\lambda)^{2}}{\alpha^{2}} \ \ \ \text{at} \ p_{1}, \ \ \
\text{whenever} \ \mu\geq0.
\end{equation*}
Hence, at the point $(p_{1},t_{1})$,
\begin{eqnarray*}
0&\leq&\psi\square_{V,\psi}\mu\\
&\leq&-\frac{2(\psi\lambda)^{2}}{\alpha^{2}}
-\psi^{2}(\boldsymbol{P}\oplus_{\alpha}\boldsymbol{Q})(\boldsymbol{\xi},
\boldsymbol{\xi})+\beta\left(-\frac{1}{4}\psi|\nabla f|^{2}
+4\frac{|\nabla\psi|^{2}}{\psi}\right)\\
&&+ \ \left[|\partial_{t}\psi|+|\Delta_{V}\psi|+2\frac{|\nabla\psi|^{2}}{\psi}
+2\sqrt{\psi}|\nabla f|\frac{|\nabla\psi|}{\sqrt{\psi}}\right]\psi\lambda\\
&\leq&-\frac{(\psi\lambda)^{2}}{\alpha^{2}}
-\psi^{2}(\boldsymbol{P}\oplus_{\alpha}\boldsymbol{Q})(\boldsymbol{\xi},
\boldsymbol{\xi})+\beta\left(-\frac{1}{4}\psi|\nabla f|^{2}+4\frac{|\nabla\psi|^{2}}{\psi}\right)\\
&&+ \ \frac{1}{2}\left(|\partial_{t}\psi|+|\Delta_{V}\psi|
+2\frac{|\nabla\psi|^{2}}{\psi}\right)^{2}+\frac{1}{2}\psi|\nabla f|^{2}
\frac{|\nabla\psi|^{2}}{\psi}.
\end{eqnarray*}
Choosing
\begin{equation}
\beta:=2\sup_{\mathcal{M}}\frac{|\nabla\psi|^{2}}{\psi}
\end{equation}
the above inequality shows that
\begin{equation*}
0\leq-\frac{(\psi\lambda)^{2}}{\alpha^{2}}
-\psi^{2}(\boldsymbol{P}\oplus_{\alpha}\boldsymbol{Q})(\boldsymbol{\xi},
\boldsymbol{\xi})+\frac{1}{2}\left(|\partial_{t}\psi|
+|\Delta_{V}\psi|+2\frac{|\nabla\psi|^{2}}{\psi}\right)^{2}
+8\sup_{\mathcal{M}}\frac{|\nabla\psi|^{4}}{\psi^{2}}
\end{equation*}
at the point $(p_{1},t_{1})$. By the properties of the cutoff function $\psi$, we arrive at
\begin{equation}
\frac{(\psi\lambda)^{2}}{\alpha^{2}}\leq\psi^{2}
(\boldsymbol{P}\oplus_{\alpha}\boldsymbol{Q})(\boldsymbol{\xi},
\boldsymbol{\xi})+8C\left(\frac{1}{T}+\frac{1+R\sqrt{K}}{R^{2}}\right)^{2}.
\label{6.41}
\end{equation}
By the same calculation as that of (\ref{6.29}), we obtain
\begin{equation*}
\psi|(\boldsymbol{P}\oplus_{\alpha}\boldsymbol{Q})(
\boldsymbol{\xi},\boldsymbol{\xi})|\leq 2m^{3/2}K_{1}(\mu_{m}+|\mu_{1}|+\beta |f|)
+4m(K_{1}+K_{2})\psi|\boldsymbol{W}|+2m\psi K_{2}
\end{equation*}
at the point $(p_{1},t_{1})$. Using $\mu_{m}+|\mu_{1}|\leq m\mu_{m}-\psi\frac{\alpha\Delta u}{u(1-f)}
-\frac{\psi|\nabla u|^{2}}{u^{2}(1-f)^{2}}+m\beta|f|$, the above estimates implies
\begin{eqnarray*}
\psi|(\boldsymbol{P}\oplus_{\alpha}\boldsymbol{Q})(\boldsymbol{\xi},
\boldsymbol{\xi})|&\leq&4m(K_{1}+K_{2})\psi|\boldsymbol{W}|
+2m\psi K_{2}+2m^{3/2}K_{1}\bigg(m\mu_{m}\\
&&- \ \psi\frac{\alpha\Delta u}{u(1-f)}
-\frac{\psi|\nabla u|^{2}}{u^{2}(1-f)^{2}}+(m+1)\beta|f|
\bigg)
\end{eqnarray*}
at the point $(p_{1},t_{1})$. Letting $a=q=0$ in Theorem 5.3, for any $\alpha\geq4$, we get
\begin{equation}
\psi\left(\frac{|\nabla u|^{2}}{u^{2}}-\alpha\frac{u_{t}}{u}\right)
\leq C n^{2}\alpha^{4}\left(\frac{1}{T}+\frac{1+R\sqrt{K}}{R^{2}}+K\right)\label{6.42}
\end{equation}
for some positive universal constant $C$, since the cutoff function is supported in a shorter cube. Using (\ref{6.42}) we have
\begin{eqnarray*}
\psi^{2}|(\boldsymbol{P}\oplus_{\alpha}\boldsymbol{Q})(\boldsymbol{\xi},
\boldsymbol{\xi})|&\leq&2m^{5/2}K_{1}\psi^{2}\lambda+2C n^{2}m^{3/2}\alpha^{4}K_{1}\psi^{2}
\left(\frac{1}{T}+\frac{1+R\sqrt{K}}{R^{2}}\right)\\
&&+ \ \left[2mK_{2}+\frac{m^{3/2}K_{1}\alpha^{2}}{2}|V|^{2}
+2C n^{2}m^{3/2}\alpha^{4}K_{1}K\right]\\
&&+ \ 4m(K_{1}+K_{2})\psi^{2}|\boldsymbol{W}|
+2m^{3/2}(m+1)K_{1}\beta|f|.
\end{eqnarray*}
According to Theorem 5.1 in \cite{ATW09} or \cite{SoupletZhang06}, we can
find a constant $C'$ depending only on $m$ so that
\begin{equation*}
\psi^{2}|\boldsymbol{W}|\leq C'\left(\frac{1}{T}+\frac{1}{R^{2}}
+K\right).
\end{equation*}
Consequently,
\begin{eqnarray*}
\psi^{2}|(\boldsymbol{P}\oplus_{\alpha}\boldsymbol{Q})(\boldsymbol{\xi},
\boldsymbol{\xi})|&\leq&2m^{\frac{5}{2}}K_{1}\psi^{2}\lambda+C''m^{\frac{3}{2}}n^{2}\alpha^{4}(K_{1}+K_{2})
\left(\frac{1}{T}+\frac{1+R\sqrt{K}}{R^{2}}\right)\\
&&+ \ \bigg[2mK_{2}+\frac{m^{3/2}K_{1}\alpha^{2}}{2}|V|^{2}
+2C m^{3/2}n^{2}\alpha^{4}K_{1}K\\
&&+ \ 4
m C'(K_{1}+K_{2})K\bigg]+4m^{5/2}K_{1}\beta|f|
\end{eqnarray*}
for another positive universal constant $C''$. Plugging it into (\ref{6.41}) implies
\begin{eqnarray*}
\frac{(\psi\lambda)^{2}}{\alpha^{2}}
&\leq& 2m^{5/2}K_{1}\psi\lambda+B_{1}\left(\frac{1}{T}+\frac{1+R\sqrt{K}}{R^{2}}
\right)
+8C\left(\frac{1}{T}+\frac{1+R\sqrt{K}}{R^{2}}\right)^{2}\\
&&+ \ B_{2}+4m^{5/2}K_{1}\beta|f|,
\end{eqnarray*}
at the point $(p_{1},t_{1})$, where
\begin{eqnarray*}
B_{1}&:=&C''m^{3/2}n^{2}\alpha^{4}(K_{1}+K_{2}),\\
B_{2}&:=&2mK_{2}+\frac{m^{3/2}K_{1}\alpha^{2}}{2}|V|^{2}
+2C m^{3/2}n^{2}\alpha^{4}K_{1}K+4
m C'(K_{1}+K_{2})K.
\end{eqnarray*}
An elementary inequality shows that
\begin{equation*}
\frac{\psi\lambda}{\alpha}
\leq 2\alpha m^{5/2}K_{1}+\sqrt{8C}\left(\frac{1}{T}+\frac{1+R\sqrt{K}}{R^{2}}
+\frac{B_{1}}{16 C}\right)+\sqrt{B_{2}}+2m^{5/4}\sqrt{K_{1}\beta|f|}
\end{equation*}
at the point $(p_{1},t_{1})$. Therefore
\begin{equation*}
\psi\lambda\leq\sqrt{8C}\alpha\left(\frac{1}{T}+\frac{1+R\sqrt{K}}{R^{2}}
+B\right)+2m^{5/4}\sqrt{K_{1}\beta|f|}
\end{equation*}
at the point $(p_{1},t_{1})$, where
\begin{equation*}
B:=\frac{2\alpha^{2}m^{5/2}K_{1}+\alpha\sqrt{B_{2}}}{\sqrt{8C}}
+\frac{B_{1}\alpha}{16C}.
\end{equation*}
As the same argument in the page 16 of \cite{HanZhang12}, using the inequality $2m^{5/4}\sqrt{K_{1}\beta|f|}\leq\beta|f|+2m^{5/2}K_{1}$ and $f<0$, we must have
\begin{equation*}
\mu\leq\sqrt{8C}\alpha\left(\frac{1}{T}+\frac{1+R\sqrt{K}}{R^{2}}
+B+\frac{2m^{5/2}K_{1}}{\sqrt{8C}\alpha}\right) \ \ \ \text{in} \ Q_{R,T}(x_{0},t_{0}).
\end{equation*}
For any unit tangent vector $\boldsymbol{\xi}$ at $x$ with $(x,t)\in Q_{R,T}(x_{0},t_{0})$, we have
\begin{equation*}
\psi\boldsymbol{V}(\boldsymbol{\xi},\boldsymbol{\xi})
\leq \sqrt{8C}\alpha\left(\frac{1}{T}+\frac{1+R\sqrt{K}}{R^{2}}+B
+\frac{2m^{5/2}K_{1}}{\sqrt{8C}\alpha}\right)(1-f) \ \ \ \text{in} \ Q_{R,T}(x_{0},t_{0}).
\end{equation*}
Taking $\alpha=4$ as in the proof of part (a), we finally obtain the following
estimate
\begin{equation*}
\psi\boldsymbol{V}(\boldsymbol{\xi},\boldsymbol{\xi})\leq C_{1}
\left(\frac{1}{T}+\frac{1+R\sqrt{K}}{R^{2}}+B'\right)(1-f) \ \ \ \text{in} \
Q_{R,T}(x_{0},t_{0})
\end{equation*}
where
\begin{equation*}
B':=C_{2}m^{5/2}n^{2}\left[K_{1}+K_{2}+\sqrt{(K_{1}+K_{2})K+K_{2}
+K_{1}|V|^{2}}\right],
\end{equation*}
for some positive universal constants $C_{1}, C_{2}$.
\\

{\bf Acknowledgments.} Yi Li is partially supported by Shanghai YangFan
Project (grant) No. 14YF1401400. The author would like to thank Shanghai Center for Mathematical
Sciences, where part of their work was done during the visit, for their hospitality.

\bibliographystyle{amsplain}

\end{document}